\documentclass[final]{siamltex}

\usepackage{amssymb,enumerate}
\usepackage{graphicx}

\newtheorem{algo}{Algorithm}

\newtheorem{rem}{Remark}[section]

\title{A weighted global GMRES algorithm with deflation for solving large Sylvester matrix equations}

\author{Najmeh Azizi Zadeh \thanks{Department of Applied Mathematics, Faculty of Mathematics $\&$ Computer Sciences, Shahid Bahonar University of Kerman, Kerman, Iran, ({\tt nazizizadeh@math.uk.ac.ir}).} \and Azita Tajaddini\thanks{Department of Applied Mathematics, Faculty of Mathematics $\&$ Computer Sciences, Shahid Bahonar University of Kerman, Kerman, Iran,
({\tt atajadini@uk.ac.ir}).} \and Gang Wu\thanks{Corresponding author. School of Mathematics, China University of Mining and Technology, Xuzhou, 221116, Jiangsu, P.R. China ({\tt gangwu76@126.com, gangwu@cumt.edu.cn}). This author is supported by the National Science Fundation of China under grant 11371176, the Natural Science Foundation of Jiangsu
Province, and the Talent Introduction Program of China
University of Mining and Technology.}}

\begin{document}

\maketitle

\begin{abstract}
The solution of large scale Sylvester matrix equation plays an important role in control and large scientific computations. A popular approach is to use the global GMRES algorithm. In this work, we first consider the global GMRES algorithm with weighting strategy, and propose some
new schemes based on residual to update the weighting matrix.
Due to the growth of memory requirements and computational
cost, it is necessary to restart the algorithm efficiently. The deflation strategy is popular
for the solution of large linear systems and large eigenvalue problems, to the best of our knowledge, little work is done
on applying deflation to the global GMRES algorithm for large Sylvester matrix equations.
We then consider how to combine the weighting strategy with deflated restarting, and propose a weighted global GMRES algorithm with deflation for solving large Sylvester matrix equations. Theoretical analysis is given to show why the new algorithm works effectively. Further, unlike the weighted GMRES-DR presented in [{\sc M. Embree, R. B. Morgan and H. V. Nguyen},
{\em Weighted inner products for GMRES and GMRES-DR},
 (2017), arXiv:1607.00255v2], we show that in our new algorithm, there is no need to change the inner product with respect to diagonal matrix to that with non-diagonal matrix, and our scheme is much cheaper. Numerical examples illustrate the numerical behavior of the proposed algorithms.
\end{abstract}



\pagestyle{myheadings}
\thispagestyle{plain}
\markboth{N. Azizi Zadeh, A. Tajaddini, and G. Wu}{\sc Weighted Global GMRES Algorithm with Deflation}

\section{Introduction}
Consider the large Sylvester matrix equation
\begin{equation}
AX+XB=C, \label{1.1}
\end{equation}
where $A \in \mathbb{R}^{n \times n}, B \in \mathbb{R}^{s \times s}$, $C \in \mathbb{R}^{n \times s}$ and $X \in \mathbb{R}^{n\times s}$, with $s\ll n$.
If we define the operator $\mathcal{A}$ as
\begin{eqnarray}\label{122}
\mathcal{A}&: & \ \mathbb{R}^{n\times s} \longrightarrow \mathbb{R}^{n\times s},\nonumber\\
X\rightarrow& & \mathcal{A}X= AX+XB.
\end{eqnarray}
Then the Sylvester matrix equation (\ref{1.1}) can be written as
\begin{equation}\label{1.2}
\mathcal{A}X=C.
\end{equation}
The Sylvester matrix equation (\ref{1.1}) plays an important role in control and communications theory, model reduction, image restoration, signal processing, filtering, decoupling techniques for ordinary partial differential equations, as well as the implementation of implicit numerical methods for ordinary differential equations; see \cite{Benner, Calvetti, Datta, Dattakn, Guennouni, Heyouni, Simonciniv}, and the references there in.
Also, (\ref{1.2}) can be rewritten as the following large linear system
\begin{equation}\label{1.3}
\big(I_s\otimes A+B^T\otimes I_n\big)vec(X)=vec(C),
\end{equation}
where $\otimes$ denotes the Kronecker product operator, and $vec(X)$ denotes the vectorize operator defined as (in MATLAB notation)
$$
vec(X)=[X(:,1); X(:,2); \ldots; X(:,s)],
$$
and $X(:, i)$ is the $i$-th column of the matrix $X\in \mathbb{R}^{n\times s}$. The linear equation systems (\ref{1.3}) have unique solution if and only if the matrix $(I_s\otimes A)+(B^T\otimes I_n)$ is nonsingular. Throughout this paper, we assume that the system (\ref{1.3}) has a unique solution. However, the size of the linear equation systems (\ref{1.3}) would be very huge. Therefore, we apply some iterative algorithms for solving (\ref{1.1}) instead of (\ref{1.3}).

There are some iterative algorithms based on the block or the matrix Krylov solvers for the solution of the Sylvester matrix equations, see, e.g. \cite{Agoujil, Dehghan, Guennouni, Heyouni, Jaimoukha, Riquet, Khorsand, Simoncinilin, Panjehali, Simoncini, Simonciniv}. The main idea behind these algorithms is to exploit the global or block (extended) Arnoldi process to construct $F$-orthonormal or orthonormal bases for the matrix or block Krylov subspaces, respectively, and then apply some projection techniques to extract approximations.

In \cite{Essai}, Essai introduced a weighted Arnoldi process for solving large linear systems. The idea is to improve and accelerate the convergence rate of the standard algorithm by constructing a $D$-orthonormal basis for the Krylov subspace, where $D$ is called the weighting matrix and is generally a positive diagonal matrix. According to \cite{Embree,Guttel}, weighting strategy can improve the algorithm by alienating the eigenvalues that obstacle the convergence. The weighting strategy has been successfully developed for solving linear systems \cite{Heyouniessai,Saberinajafi}, matrix equations \cite{Panjeh,Panjehbeik}, and large eigenvalue problems \cite{Zhong}. For example,
Mohseni Moghadam {\it et al.} \cite{Panjeh} presented a weighted global FOM method for solving nonsymmetric linear systems. They used the Schur complement formula and a new matrix product, and gave some theoretical results to show rationality of the proposed algorithm. In \cite{Panjehbeik}, Panjeh Ali Beik {\it et al.} proposed weighted global FOM and weighted global GMRES algorithms for solving the general coupled linear matrix equations.

For the sake of the growth of memory requirements and computational
cost, the global Krylov subspace algorithms will become
impractical as the step of the global Arnoldi process proceeds.
For Krylov subspace algorithms,
one remedy is to use some
restarting strategies \cite{Saad}.
A popular restarting strategy is the
deflated restarting (also refer to as thick-restarting or deflation) strategy advocated in \cite{Jiang,Morganr,Morgan,Mor,Wu,KWu,Zhong}, in which the approximate eigenvectors are put
firstly in the search subspace. Here ``deflated restarting" (or deflation) means computing some approximate eigenvectors corresponding
to some eigenvalues, and using them to
``deflate" these eigenvalues from the spectrum of the matrix, to speed up the convergence of the iterative algorithm. The deflation strategy is popular
for the solution of large linear systems \cite{Morganr,Morgan} and large eigenvalue problems \cite{Jiang,Mor,Wu,KWu,Zhong}, to the best of our knowledge, little work is done
on applying the deflated restarting strategy on the global GMRES algorithm for large Sylvester matrix equations.

In this paper, we try to fill in this gap. As was pointed out in \cite{Embree,Essai,Guttel}, the optimal choice of the weighting matrix $D$ in the weighted approaches is still an open problem and needs further investigation. We first apply the weighting strategy to the global GMRES algorithm, and present three new schemes to update the weighting matrix at each restart.
To accelerate the convergence of the weighted global GMRES algorithm, we consider how to knit the deflation strategy together with it, and the key is that the Sylvester matrix equation can be rewritten as a linear system of the form (\ref{1.2}) theoretically. The new algorithm can be understood as applying the deflation technique to remove some small eigenvalues of the matrix $(I_s\otimes A+B^T\otimes I_n)$ at each restart. Theoretical results and numerical experiments show that the weighting strategy with deflation can produce iterations that give faster convergence than the conventional global GMRES algorithm, and a combination of these two strategies is more efficient and robust than its two original counterparts.

This paper is organized as follows. After presenting the weighted global GMRES algorithm for the solution of Sylvester matrix equations in section 2, the deflated version of this algorithm is established in section 3.
Some numerical experiments confirm the superiority of our new algorithm over the conventional ones in section 4.


\section{A weighted global GMRES algorithm for large Sylvester matrix equations}\label{sec2}
In this section, we recall some notations and definitions that will be used in this paper, and briefly introduce the weighted global Arnoldi process as well as the weighted global GMRES algorithm. Specifically, we propose three new schemes based on residual to update the weighting matrix during iterations.



The global generalized minimal residual (GLGMRES) algorithm is well-known for solving linear systems with multiple right-hand sides and for matrix equations \cite{Jbilou,Panjehali}, which is an oblique projection method based on matrix Krylov subspace.
Let us introduce the {\it weighted} global GMRES algorithm for Sylvester matrix equations. Let $D=diag( d_1,d_2, \ldots, d_n)$ be a diagonal matrix with $d_i>0,~i=1, 2,\ldots,n$, and let ${\bf u}, {\bf v} \in \mathbb{R}^{n}$ be given, then the $D$-inner product with respect to two vectors is defined as \cite{Zhong}
\begin{eqnarray*}
({\bf u}, {\bf v})_D={\bf v}^{T}D{\bf u} &=\sum_{i=1}^{n}d_i{\bf u}_i{\bf v}_j,
\end{eqnarray*}
and the associated $D$-norm of $\bf u$ is defined as
\begin{eqnarray*}
\Vert {\bf u}\Vert_D & =\sqrt{({\bf u},{\bf u})_D}, \quad \forall {\bf u}\in \mathbb{R}^{n}.
\end{eqnarray*}
For two matrices $Y, Z\in \mathbb{R}^{n\times s}$, the $D$-inner product is defined as \cite{Panjeh}
\begin{eqnarray*}
(Y,Z)_D= & trace(Z^{T}DY),
\end{eqnarray*}
where $trace(\cdot)$ denotes the trace of a matrix, and $Z^T$ represents the transpose of the matrix $Z$.
It can be verified that \cite{Panjeh}
$$
(Y,Z)_D=\big(vec(Y),vec(Z)\big)_{I_s \otimes D}.
$$
Also, the $D$-norm associated with this inner product is
\begin{eqnarray*}
\Vert Y \Vert_D & =\sqrt{(Y,Y)_D}, \quad \forall  \ Y\in \mathbb{R}^{n\times n}.
\end{eqnarray*}
Next we introduce a useful product that will be used latter:
\begin{definition} \cite{Panjeh}\label{def1}
Let $A=[A_1,A_2,\ldots, A_p]$ and $B=[B_1,B_2,\ldots,B_l]$, where $A_i, B_j \in \mathbb{R}^{n\times s},~i=1,2,\ldots,p,~j=1,2,\ldots,l$. Then elements of the matrix $A^{T} \diamond_D B$ is defined as
\begin{eqnarray}\label{2.1}
(A^{T} \diamond_D B)_{ij}=(A_i, B_j)_D, \quad i=1, 2, \ldots, p,~j=1, 2, \ldots, l.
\end{eqnarray}
\end{definition}

It was shown that $A^{T} \diamond_D B=A^{T} \diamond (D B)$ \cite{Panjeh}.
Let $V\in \mathbb{R}^{n\times s}$ be an initial block vector that is $D$-orthogonal, that is, orthonormal with respect to the $\diamond_D$-inner product. The following algorithm presents an $m$-step weighted global Arnoldi process \cite{Panjeh}.
\medskip

\begin{algo}
{\textit{ The $m$-step weighted global Arnoldi process}}
\label{alg1}
\begin{enumerate}
  \item \textbf{Input:} $A \in \mathbb{R}^{n\times n}$, $B \in \mathbb{R}^{s\times s}$, $V \in \mathbb{R}^{n\times s}$, a positive diagonal matrix $D$ and an  integer number $m>0$.
  \item \textbf{Output:} $\mathcal{V}_m=[V_1, V_2, \ldots,V_m], \bar{H}_m=[h_{i,j}].$
  \item Compute $\beta=\Arrowvert V\Arrowvert_{D}$, $V_{1}=V/\beta.$
  \item for {$j=1,2,\ldots, m$ }
  \item ~~Compute $W=\mathcal{A}V_j=AV_{j}+V_jB.$
  \item ~~for {$i=1,2,\ldots,j$}
  \item ~~~~$h_{ij}=(W, V_{i})_{D}.$
  \item ~~~~$W=W-h_{ij}V_{i}.$
  \item ~~end
  \item ~~Compute  $h_{j+1, j}=\Arrowvert W\Arrowvert_{D}$. If $h_{j+1,j}=0$ then stop.
  \item ~~Compute $V_{j+1}=W/h_{j+1, j}.$
  \item end
\end{enumerate}
\end{algo}
\medskip
The weighted global Arnoldi process constructs a $D$-orthonormal basis $\mathcal{V}_{m}=[V_1,V_2,\ldots ,V_{m}]\in \mathbb{R}^{n\times ms}$, i.e.,
\begin{eqnarray*}
(V_i, V_{j})_D & =\delta_{ij}, \quad \forall i,j=1,2, \ldots, m,
\end{eqnarray*}
for the matrix Krylov subspace
\begin{eqnarray*}
\mathcal{K}_m (\mathcal{A}, V_1) &:=& span\lbrace V_1, \mathcal{A}V_1, \ldots, \mathcal{A}^{m-1}V_1 \rbrace \\
&=&\Big \lbrace \sum_{i=1}^{m-1}\alpha_{i}\mathcal{A}^{i-1}V_1 \vert \alpha_{i} \in \mathbb{R}, i=1,2, \ldots, m-1 \Big \rbrace.
\end{eqnarray*}

Let $\bar{H}_m=\Big [\begin{array}{c}
                                                                                                    H_m \\
                                                                                                    h_{m+1,m}e_m^T
                                                                                                  \end{array} \Big ]
 \in \mathbb{R}^{(m+1)\times m}$ be a quasi-upper Hessenberg matrix whose nonzeros entries $h_{ij}$ are defined by Algorithm \ref{alg1}, and the matrix $H_m$ is obtained from the matrix $\bar{H}_m$ by deleting its last row. Note that the matrix $\mathcal{V}_{m}$ is $D$-orthogonal.
With the help of Definition \ref{def1}, we obtain the following relations
\begin{eqnarray}\label{2.2}
[\mathcal{A}V_1, \mathcal{A}V_2,\ldots, \mathcal{A}V_m]&=&\mathcal{V}_{m}(H_m\otimes I_s)+h_{m+1,m} V_{m+1} ({\bf e}_m^{T}\otimes I_s),\\ \label{2.3}
&=&\mathcal{V}_{m+1}(\bar{H}_m\otimes I_s),\\ \label{2.4}
\bar{H}_m&=&\mathcal{V}_{m+1}^{T}\diamond_{D} \mathcal{A}\mathcal{V}_m,
\end{eqnarray}
where $\mathcal{V}_{m+1}=[V_1,V_2,\ldots ,V_{m+1}]\in \mathbb{R}^{n\times (m+1)s}$. Define
\begin{equation}\label{255}
\mathcal{A}\mathcal{V}_m\equiv[\mathcal{A}V_1, \mathcal{A}V_2,\ldots, \mathcal{A}V_m],
\end{equation}
then (\ref{2.2}) can be rewritten as
\begin{eqnarray}\label{2.5}
\mathcal{A}\mathcal{V}_m=\mathcal{V}_{m+1}(\bar{H}_m\otimes I_s).
\end{eqnarray}

We are in a position to consider the weighted global GMRES algorithm for solving (\ref{1.1}). Let $X_0 \in \mathbb{R}^{n\times s} $ be the initial guess, and the initial residual be
$R_0=C-AX_0-X_0B$.
In the weighted global GMRES algorithm, we construct an approximate solution of the form
\begin{equation}\label{2.8}
X_m=X_0+\mathcal{V}_m({\bf y}^{w}_m\otimes I_s),
\end{equation}
where ${\bf y}^{w}_m \in \mathbb{R}^m$. The corresponding residual is
\begin{eqnarray}\nonumber
R_m&=&C-AX_m-X_mB\\ \nonumber
&=&(C-AX_0-X_0B)-\big(A\mathcal{V}_m({\bf y}^{w}_m\otimes I_s)+\mathcal{V}_m({\bf y}^{w}_m\otimes I_s)B\big),\\
&=&R_0-\mathcal{A}\mathcal{V}_m({\bf y}^{w}_m\otimes I_s), \label{2.9}
\end{eqnarray}
here we used (\ref{122}) and (\ref{255}). Substituting (\ref{2.3}) into (\ref{2.9}), we arrive at
\begin{eqnarray}\nonumber
R_m&=&\beta V_1-\mathcal{V}_{m+1}(\bar{H}_m\otimes I_s)({\bf y}^{w}_m\otimes I_s),\\
&=&\mathcal{V}_{m+1}\big((\beta {\bf e}_{1}-\bar{H}_m{\bf y}^{w}_m) \otimes I_s\big), \label{2.10}
\end{eqnarray}
where ${\bf e}_1$ is the first canonical basis vector in $\mathbb{R}^{m+1}$. Note that the residual is $D$-orthogonal to $\mathcal{A}\mathcal{K}_m(\mathcal{A},R_0)$, i.e.,
\begin{equation}\label{2.7}
R_m=C-\mathcal{A}X_m\perp_D \mathcal{A}\mathcal{K}_m(\mathcal{A},R_0),
\end{equation}
where $\mathcal{A}\mathcal{K}_m(\mathcal{A},R_0)=span\lbrace \mathcal{A}R_0, \ldots, \mathcal{A}^{m}R_0 \rbrace$, and ``$\perp_D$" means orthgonal with respect to the ``$\diamond_D$" inner product.
%

In order to compute ${\bf y}^{w}_m$, we have from (\ref{2.10}) and (\ref{2.7}) that
\begin{eqnarray*}
\min_{\bf y} \Vert R_m\Vert_D&=&\min_{\bf y} \Vert \mathcal{V}_{m+1}((\beta {\bf e}_{1}-\bar{H}_m{\bf y}) \otimes I_s)\Vert_D\\
&=&\min_{\bf y} \Vert (\beta {\bf e}_{1}-\bar{H}_m{\bf y})^T(\mathcal{V}_{m+1})^TD(\mathcal{V}_{m+1})(\beta {\bf e}_{1}-\bar{H}_m{\bf y})\Vert_2 \\
&=&\min_{\bf y} \Vert \beta {\bf e}_{1}-\bar{H}_m{\bf y}\Vert_2,
\end{eqnarray*}
where we used $\mathcal{V}_{m+1}^{T}\diamond_D \mathcal{V}_{m+1}=I_{m+1}$. Thus,
\begin{equation}\label{2.11}
{\bf y}^{w}_m={ \rm arg min}_{{\bf y}^{w}\in \mathbb{R}^{m}}\Vert \beta {\bf e}_1 -\bar{H}_m {\bf y} ^{w}\Vert_2,
\end{equation}
or equivalently,
\begin{equation}\label{2.12}
\bar{H}_m^{T}\bar{H}_m{\bf y}^{w}_m=\bar{H}_m^{T}\beta{\bf e}_1.
\end{equation}

As was pointed out in \cite{Essai,Heyouniessai,Zhong}, the optimal choice of $D$ in the weighted approaches is still an open problem and needs further investigation.
Some choices for the weighting matrix have been considered in, say, \cite{Heyouniessai, Panjehbeik, Saberinajafi}.
Also, to speed up the convergence rate, it was suggested to use a weighted inner product that changes at each restart \cite{Embree,Guttel}.
In this section, we propose three choices based on the residual $R_m$, which could be updated during iterations:
\medskip
\begin{description}
  \item[Option 1:] Let $\Vert R(:,t)\Vert_2=\max \lbrace \Vert R_m(:,i)\Vert_2, ~  1\leq i\leq s \rbrace$,
  where $R_m(:,t)$ is the $t$-th column of residual matrix $R_m$. Then we define
    $D_1=diag\left(\frac{\vert R(:,t)\vert}{\Vert R(:,t)\Vert_2}\right)$, where $\vert R(:,t)\vert$ stands for the absolute value of $R(:,t)$.
  \item[Option 2:] Similarly, let $\Vert R(:,t)\Vert_2=\min \lbrace \Vert R_m(:,i)\Vert_2,~1\leq i\leq s \rbrace$, then we define $D_2=diag\left(\frac{\vert R(:,t)\vert}{\Vert R(:,t)\Vert_2}\right)$.
  \item[Option 3:] Use the mean of the block residual $R_m$, i.e, $ D_3=diag\left(\vert(\frac{\sum_{i=1}^s R_m(:,i)}{s})\vert\right)$.
\end{description}

Combining these weighting strategies with Algorithm \ref{alg1}, we have the following algorithm.

\begin{algo}
{\textit{ A restarted weighted global GMRES algorithm for large Sylvester matrix equations\quad (W-GLGMRES)}}
\label{alg2}
\begin{enumerate}
  \item \textbf{Input:} $A\in \mathbb{R}^{n\times n}$, $B \in \mathbb{R}^{s\times s}$, $C \in \mathbb{R}^{n\times s}$. Choose the initial guess, $X_0 \in \mathbb{R}^{n\times s}$, an initial positive diagonal matrix D and an integer $m>0$, and the tolerance $tol>0$.
  \item \textbf{Output:} The approximation $X_m$.
  \item Compute $R_0=C-AX_0-X_0B$ and $\beta=\Arrowvert R_0\Arrowvert_{D}$, $V=R_0/\beta.$
  \item Run Algorithm \ref{alg1} with the initial block vector $V$ to obtain the matrices $\mathcal{V}_m, \bar{H}_m$.
  \item Solving $\min_{{\bf y}^{w}\in \mathbb{R}^{m}} \Vert \beta e_1 -\bar{H}_m{\bf y}^{w}\Vert_2$ for ${\bf y}^{w}_{m}$.
  \item Compute $X_m=X_0+\mathcal{V}_m({\bf y}^w_m\otimes I_s)$ and $R_m=C-AX_m-X_mB$. If $\Vert R_m\Vert_D< tol$, then stop, else continue.
  \item Set $X_0=X_m$ and update the weighting matrix $D$ according to Options 1--3, and go to Step 3.
\end{enumerate}
\end{algo}
\medskip
\begin{rem}
As was mentioned before, the Sylvester matrix equation (\ref{1.1}) can be reformulated as the linear system (\ref{1.3}). Thus, the three choices of $\{D_j\}_{j=1}^3$ for (\ref{1.1}) can be understood as
the weighted GMRES algorithm with the weighting matrices
\begin{eqnarray*}
\widehat{D}_j=I_s \otimes D_j, \  j=1,2,3,
\end{eqnarray*}
respectively, for solving the linear system (\ref{1.3}). The theoretical results and discussions given in \cite{Embree, Guttel} on weighted GMRES for large linear systems apply here  trivially, and one refers to \cite{Embree, Guttel} for why the weighted strategy can speed up the computation. This also interprets why the weighted strategy can improve the numerical performance of the standard global GMRES; see the numerical experiments made in Section \ref{sec4}.
\end{rem}
\section{A weighted global GMRES with deflation for large Sylvester matrix equations}\label{sec3}
In this section, we speed up the weighted global GMRES algorithm by using
the deflated restarting strategy that is popular for large linear systems and large eigenvalue problems \cite{Jiang,Morganr,Morgan,Mor,KWu,Zhong}.
In the first cycle of the weighted global GMRES algorithm with deflation, the standard weighted global GMRES algorithm is run. To apply the deflated restarting strategy,
we need to compute $k~(1\leq k\leq m)$ weighted harmonic Ritz pairs. Let $\mathcal{V}_m$ be the $D$-orthonormal basis obtained from the ``previous" cycle, we seek $k$ weighted harmonic Ritz pairs $(\theta_i,Y_i)$ that satisfy
\begin{equation}\label{3.1}
\mathcal{A}\mathcal{V}_mY_i-\theta_i\mathcal{V}_m Y_i\perp_D (\mathcal{A}-\sigma I)\mathcal{K}_m(\mathcal{A},V),\quad i=1,\ldots,k,
\end{equation}
where
$$
\mathcal{A}\mathcal{V}_m=[\mathcal{A}V_1, \mathcal{A}V_2,\ldots, \mathcal{A}V_m]=[AV_1+V_1B,\ldots,AV_m+V_mB],
$$
and $Y_i={\bf g}_i\otimes I_s \in \mathbb{C}^{ms\times s}$ with ${\bf g}_i \in \mathbb{C}^{m\times 1}$.
In this work, we want to deflate some smallest eigenvalues in magnitude, and a shift $\sigma=0$ is used throughout this paper.

From (\ref{2.5}), we have that
\begin{eqnarray*}\nonumber
\mathcal{A}\mathcal{V}_m({\bf g}_i\otimes I_s)-\theta_i \mathcal{V}_m({\bf g}_i\otimes I_s) &=& (\mathcal{A}\mathcal{V}_m-\theta_i \mathcal{V}_m)({\bf g}_i\otimes I_s)\\
 & =&\mathcal{V}_{m+1} \left((\bar{H}_m-\theta_i\bar{I}_m)\otimes I_s \right)({\bf g}_i\otimes I_s), \label{3.2}
\end{eqnarray*}
where $\bar{I}_m=\left [\begin{array}{c}
I_m \\
0
\end{array}\right ].$
By (\ref{3.1}) and Definition \ref{def1}, we can compute $(\theta_i,{\bf g}_i)$ via solving the following (small-sized) generalized eigenvalue problem
\begin{equation}\label{3.3}
\theta_i\big((\mathcal{A}\mathcal{V}_m)^T \diamond_{D} \mathcal{V}_m\big){\bf g}_i=(\mathcal{A}\mathcal{V}_m)^T \diamond_{D} (\mathcal{A}\mathcal{V}_m){\bf g}_i.
\end{equation}
From (\ref{2.5}) and the fact that $\mathcal{V}_{m+1}^{T}\diamond_{D}\mathcal{V}_{m+1}=I_{m+1}$, we rewrite (\ref{3.3}) as
\begin{equation}\label{3.4}
\theta_iH_m^T {\bf g}_i=\bar{H}_m^T\bar{H}_m{\bf g}_i.
\end{equation}
If $H_m$ is nonsingular, (\ref{3.4}) is equivalent to
\begin{equation}\label{3.5}
(H_m+h_{m+1,m}^2H_m^{-T}{\bf e}_m{\bf e}_m^T){\bf g}_i=\theta_i{\bf g}_i.
\end{equation}

Then we define the ``weighted harmonic Ritz vector" as $Y_i={\bf g}_i\otimes I_s$, and the corresponding harmonic residual is
\begin{eqnarray*}
\widetilde{R}_i=\mathcal{A}\mathcal{V}_mY_i - \theta_i\mathcal{V}_mY_i&=&\mathcal{V}_{m+1}\big((\bar{H}_m-\theta_i\bar{I}_m){\bf g}_i\otimes I_s\big)\\
&=&\mathcal{V}_{m+1}({\bf \widetilde{r}}_i\otimes I_s),\quad i=1,\ldots,k,
\end{eqnarray*}
where ${\bf \widetilde{r}}_i=(\bar{H}_m-\theta_i\bar{I}_m){\bf g}_i,~ i=1,\ldots,k$.
\begin{rem}
In \cite{Duan}, a global harmonic Arnoldi method was proposed for computing harmonic Ritz pairs
of large matrices. Here the difference is that our approach is based on the weighted projection techniques, and the method in \cite{Duan} is a special case of ours as $D=I$.
\end{rem}

In the following, we characterize the relationship between the weighted harmonic residuals and the residual from the weighted global GMRES algorithm.
Let $X_0\in \mathbb{R}^{n\times s}$ be the initial guess and $R_0=C-AX_0-X_0B$ be the initial residual.
After one cycle of the weighted global GMRES Algorithm, we have the approximate solution $X_m=X_0+\mathcal{V}_m({\bf y}^{w}_m\otimes I_s)$, where ${\bf y}^{w}_m$ is defined in (\ref{2.11}). The associated residual with respect to $X_m$ is
$$
R_m=\mathcal{V}_{m+1}\big((\beta {\bf e}_1-\bar{H}_m{\bf y}^{w}_m)\otimes I_s\big)=\mathcal{V}_{m+1}({\bf {r}}_m\otimes I_s),
$$
where ${\bf {r}}_m\equiv\beta {\bf e}_1-\bar{H}_m{\bf y}^{w}_m$.
The following result shows that ${\bf {r}}_m$ and $\{{\bf \widetilde{r}}_i\}_{i=1}^k$ are collinear with each other.
\begin{theorem}\label{pro3.1}
Let $\widetilde{R}_i=\mathcal{V}_{m+1}({\bf \widetilde{r}}_i\otimes I_s),~i=1,\ldots,k$, be the weighted harmonic residuals and $R_m=\mathcal{V}_{m+1}({\bf {r}}_m\otimes I_s)$ be the weighted global GMRES residual, respectively, where ${\bf \widetilde{r}}_i=(\bar{H}_m-\theta_i\bar{I}_m){\bf g}_i$ and $\ {\bf {r}}_m=\beta {\bf e}_1-\bar{H}_m{\bf y}^{w}_m.$ Then $\{{\bf \widetilde{r}}_i\}_{i=1}^k$ and ${\bf {r}}_m$ are collinear with each other.
\end{theorem}
\begin{proof}
Note that both the weighted harmonic residuals and the residual of the weighted global GMRES algorithm are in $range(\mathcal{V}_{m+1})$, and they are both $D$-orthogonal to $\mathcal{A}\mathcal{V}_m$. Thus, there is an $s\times s$ matrix $T$, such that
\begin{equation}\label{3.6}
\mathcal{V}_{m+1}({\bf \widetilde{r}}_i\otimes I_s)=\mathcal{V}_{m+1}({\bf {r}}_m\otimes I_s)T,
\end{equation}
so we have
$({\bf \widetilde{r}}_i\otimes I_s)=({\bf {r}}_m\otimes I_s)T$. Let ${\bf {r}}_m=(\eta_1,\eta_2,\ldots,\eta_{m+1})^T$ and ${\bf \widetilde{r}}_i=(\tau_1,\tau_2,\ldots,\tau_{m+1})^T$, we obtain from (\ref{3.6}) that
$$
({\bf {r}}_m\otimes I_s)T=\left (\begin{array}{c}
                                         \eta_1I_s \\
                                         \eta_2I_s \\
                                         \vdots \\
                                         \eta_{m+1}I_s
                                       \end{array}\right )T=\left (\begin{array}{c}
                                         \eta_1T \\
                                         \eta_2T \\
                                         \vdots \\
                                         \eta_{m+1}T
                                       \end{array}\right )=\left (\begin{array}{c}
                                         \tau_1I_s \\
                                         \tau_2I_s \\
                                         \vdots \\
                                         \tau_{m+1}I_s
                                       \end{array}\right ),
$$
which implies that $\{{\bf \widetilde{r}}_i\}_{i=1}^k$ and ${\bf {r}}_m$ are collinear with each other.
\end{proof}

We are ready to consider how to apply the deflated restarting strategy to the weighted global GMRES algorithm, and show that a weighted global Arnoldi-like relation still holds after restarting. Let $G_k=[{\bf g}_1, {\bf g}_2, \ldots, {\bf g}_k]$, and let $G_k=Q_k\Gamma_k$ be the reduced QR factorization.
We stress that both forming $G_k$ and computing the QR decomposition can be implemented in real arithmetics; see Step 9 of Algorithm \ref{alg3}.
Then we orthonormalize ${\bf {r}}_m$ against $\left[\begin{array}{c}
Q_k \\
0_{1\times k}
\end{array} \right]$ to get ${\bf q}_{k+1}$, and let $Q_{k+1}=\left[ \begin{array}{c|c}
\begin{array}{c}
Q_k \\
0_{1\times k}
\end{array} & {\bf q}_{k+1}
\end{array}\right]$.

Notice that both (\ref{2.2}) and (\ref{2.3}) hold in the first cycle, i.e.,
$$\mathcal{A}\mathcal{V}_m=\mathcal{V}_{m}(H_m\otimes I_s)+h_{m+1,m} V_{m+1} ({\bf e}_m^{T}\otimes I_s),$$
and thus
{\small\begin{equation}\label{3.7}
\mathcal{A}\mathcal{V}_m({\bf g}_i\otimes I_s)=\mathcal{V}_{m}(H_m\otimes I_s)({\bf g}_i\otimes I_s)+h_{m+1,m}V_{m+1}({\bf e}_m^{T}{\bf g}_i\otimes I_s), \ i=1,\ldots,k .
\end{equation}}
That is,
$$\mathcal{A}\mathcal{V}_m({\bf g}_i\otimes I_s)=[\mathcal{V}_m \ V_{m+1}]\left[\begin{array}{c}
H_m{\bf g}_i\\
h_{m+1,m}{\bf e}_m^{T}{\bf g}_i
\end{array}\right] \otimes I_s.$$
We note that
\begin{eqnarray*}
{\bf \widetilde{r}}_i=(\bar{H}_m-\theta_i\bar{I}_m){\bf g}_i
=\left[\begin{array}{c}
(H_m-\theta_iI){\bf g}_i\\
h_{m+1,m}{\bf e}_m^{T}{\bf g}_i
\end{array}\right],
\end{eqnarray*}
and
\begin{eqnarray*}
\left[\begin{array}{c}
H_m{\bf g}_i\\
h_{m+1,m}{\bf e}_m^{T}{\bf g}_i
\end{array}\right]&=&\left[\begin{array}{c}
H_m{\bf g}_i-\theta_i{\bf g}_i+\theta_i{\bf g}_i\\
h_{m+1,m}{\bf e}_m^{T}{\bf g}_i
\end{array}\right]\\
&=&\theta_i\left[\begin{array}{c}
                                    {\bf g}_i \\
                                     0
                                   \end{array}\right]+{\bf \widetilde{r}}_i, \ i=1,\ldots,k.
\end{eqnarray*}

As a result,
$$\left[\begin{array}{c}
H_m{\bf g}_i\\
h_{m+1,m}{\bf e}_m^{T}{\bf g}_i
\end{array}\right]\in span \left\{ \left[\begin{array}{c}
                                    G_k \\
                                     0
                                   \end{array}\right], {\bf {r}}_m\right\}=span\{Q_{k+1}\}, \ i=1,\ldots,k,$$
where we used $\{{\bf \widetilde{r}}_i\}_{i=1}^k$ and ${\bf {r}}_m$ are collinear with each other; see Theorem \ref{pro3.1}.
Therefore,
$$\mathcal{A}\mathcal{V}_m({\bf g}_i\otimes I_s)\in span\{ \mathcal{V}_{m+1}(Q_{k+1}\otimes I_s)\}, \ \ i=1,\ldots,k,$$
and
\begin{eqnarray}\label{377}
\mathcal{A}\mathcal{V}_m(Q_k\otimes I_s)\subseteq span\{ \mathcal{V}_{m+1}(Q_{k+1}\otimes I_s)\}.
\end{eqnarray}
Define $\mathcal{V}_k^{new}\equiv\mathcal{V}_m(Q_k\otimes I_s)=[V_1^{new},\ldots,V_k^{new}]$, where $V_i^{new}\in\mathbb{R}^{n\times s},~1\leq i\leq k$, so we have
$$
\mathcal{A}\mathcal{V}_m(Q_k\otimes I_s)=\mathcal{A}\mathcal{V}_k^{new}=\mathcal{A}[V_1^{new},\ldots,V_k^{new}].
$$
If we denote
\begin{eqnarray*}
\mathcal{V}_{k+1}^{new}&=&\mathcal{V}_{m+1}(Q_{k+1}\otimes I_s),
\end{eqnarray*}
then we have from (\ref{377}) that
$$\mathcal{A}\mathcal{V}_m(Q_k\otimes I_s)\subseteq span\{\mathcal{V}_{k+1}^{new}\},$$
and there is a $(k+1)s\times ks$ matrix $P$ such that
$$\mathcal{A}\mathcal{V}_m(Q_k\otimes I_s)=\mathcal{V}_{m+1}(\bar{H}_m\otimes I_s)(Q_{k}\otimes I_s)=\mathcal{V}_{m+1}(Q_{k+1}\otimes I_s)P=\mathcal{V}_{k+1}^{new}P.$$
The condition $\mathcal{V}_{k+1}^{new^T}\diamond_{D}\mathcal{V}_{k+1}^{new}=I_{k+1}$ yields
\begin{eqnarray*}
P&=&(Q_{k+1}\otimes I_s)^T\mathcal{V}_{m+1}^T\diamond_D \mathcal{A}\mathcal{V}_m(Q_k\otimes I_s)\\
&=&(Q_{k+1}\otimes I_s)^T\mathcal{V}_{m+1}^T\diamond_D \mathcal{V}_{m+1}(\bar{H}_m\otimes I_s)(Q_k\otimes I_s),\\
&=&Q_{k+1}^T\bar{H}_mQ_k\otimes I_s=\bar{H}_k^{new}\otimes I_s,
\end{eqnarray*}
where $\bar{H}_k^{new}\equiv Q_{k+1}^T\bar{H}_mQ_k\in\mathbb{R}^{(k+1)\times k}$ is generally a dense matrix. In conclusion, we obtain
\begin{eqnarray}\label{3.8}
\mathcal{A}\mathcal{V}_k^{new}=\mathcal{V}_{k+1}^{new}(\bar{H}_k^{new}\otimes I_s).
\end{eqnarray}

Then, we run the weighted global Arnoldi Algorithm from index $(k+1)$ to $m$ with the last $s$ columns of $\mathcal{V}_{k+1}^{new}$ as the starting matrix to build a new $m$-step global Arnoldi relation.
However, $D$ is updated in our new algorithm, see Step 11 of Algorithm \ref{alg3}. In other words, we will perform the remaining $m-k$ steps of global Arnoldi process with a new $D$ (denoted by $D^{new}$ that is from the ``current" residual $R_m$) after deflated restarting.

Denote $\mathcal{V}_m=[\mathcal{V}_{k+1}^{new}~|~\mathcal{V}_{\perp}^{new}]$, where $\mathcal{V}_{\perp}^{new}=[V_{k+2}^{new},\ldots,V_{m}^{new}]$, then $\mathcal{V}_{\perp}^{new}$ is $D^{new}$-orthogonal and it is also $D^{new}$ orthogonal to $\mathcal{V}_{k+1}^{new}$. Denote by $D^{old}$ the weighting matrix obtained from the residual of the ``previous" cycle, we define $\widetilde{D}$-orthogonality of $\mathcal{V}_m$ as follows
\begin{eqnarray}\label{399}
\mathcal{V}_{m}^{T}\diamond_{\widetilde{D}}\mathcal{V}_{m}\equiv\left\{\begin{array}{lll}
\mathcal{V}_{k+1}^{new^T}\diamond_{D^{old}}\mathcal{V}_{k+1}^{new}&=&I_{k+1}, \vspace{0.2cm} \\
\mathcal{V}_{\perp}^{new^T}\diamond_{D^{new}}\mathcal{V}_{\perp}^{new}&=&I_{m-k-1}, \vspace{0.2cm} \\
\mathcal{V}_{\perp}^{new^T}\diamond_{D^{new}}\mathcal{V}_{k+1}^{new}&=&0.
\end{array}\right.
\end{eqnarray}
That is, $\mathcal{V}_{m}^{T}\diamond_{\widetilde{D}}\mathcal{V}_{m}=I_m$, and a global Arnoldi-like relation still holds after restarting.
\begin{rem}
In the weighted GMRES-DR presented in \cite[pp.20]{Embree}, Embree {\it et al.} showed how to restart the weighted GMRES-DR algorithm with a change of inner product by using the Cholesky factorization.
However, the new weighting matrix is non-diagonal any more in their strategy, and the computational cost will become much higher when computing the $D$-inner products with respect to non-diagonal matrices. Thanks to (\ref{399}), we indicate that without changing inner products in the weighted and deflated restarting algorithms, a global Arnoldi-like relation is still hold. So there is no need to change the inner product with respect to diagonal matrix to that with non-diagonal matrix, and our scheme is cheaper.
\end{rem}

In summary, we have the following theorem.

\begin{theorem}\label{th3.2}
A global Arnoldi-like relation holds for the weighted global Arnoldi algorithm with deflation
\begin{eqnarray}\label{3.9}
\mathcal{A}\mathcal{V}_m&=&\mathcal{V}_{m}(H_m\otimes I_s)+h_{m+1,m} V_{m+1} ({\bf e}_m^{T}\otimes I_s),\\
&=&\mathcal{V}_{m+1}(\bar{H}_m\otimes I_s). \label{3.10}
\end{eqnarray}
\end{theorem}
We are ready to present the main algorithm of this paper.
\begin{algo}
{\textit{ A weighted global GMRES with deflation for large Sylvester matrix equations (W-GLGMRES-D) }}
\label{alg3}
\begin{enumerate}
  \item \textbf{Input:} { $A\in \mathbb{R}^{ n\times n}, B\in \mathbb{R}^{ s\times s}, C\in \mathbb{R}^{ n\times s}$. Choose an initial guess $X_0\in \mathbb{R}^{ n\times s},$ a positive diagonal matrix $D$, the positive integer number $m$ and a convergence tolerance $tol>0$.}
  \item \textbf{Output:} The approximation $X_m$.
  \item  Compute $R_0=C-AX_0-X_0B$, $\beta=\Arrowvert R_0\Arrowvert_{D}$ and $V_{1}=R_0/\beta.$
  \item  Set ${\bf c}=[\beta; 0_{m\times 1}].$
  \item  Run the weighted global Arnoldi algorithm to obtain $\mathcal{V}_{m+1}$ and $\bar{H}_m$.
  \item Solve $\min_{{\bf y}^{w}} \Vert {\bf c} -\bar{H}_m{\bf y}^{w}\Vert_2$ for ${\bf y}^{w}_m$.
  \item  Compute $X_m=X_0+\mathcal{V}_m({\bf y}^{w}_m\otimes I_s)$, and $R_m=C-AX_m-X_mB$.
  \item  If $\Vert R_m\Vert_D< tol$, then stop, else continue.
  \item Compute the $k$ eigenpairs $(\theta_i, {\bf g}_i)$ with the smallest magnitude from (\ref{3.4}) or (\ref{3.5}). Set $G_k=[{\bf g}_1,\ldots ,{\bf g}_k]$: We first sperate the $\{{\bf g}_i\}'$s into real and imaginary parts if they are complex, to form the columns of $G_k\in \mathbb{R}^{m\times k}$. Both the real and the imaginary parts need to be included. Then we compute the reduced QR factorization of $G_k$ : $G_k=Q_k\Gamma_k$, where $Q_k=[{\bf q}_1,\ldots ,{\bf q}_k]$. Notice that both $G_k$ and $Q_k$ are real.
  \item {Extend the vectors ${\bf q}_1,\ldots ,{\bf q}_k$ to length $m+1$ with zero entries, then orthonormalize the vector ${\bf r}_m={\bf c} -\bar{H}_m{\bf y}^{w}_m$ against the columns of $\left[ \begin{array}{c}
Q_k \\
0
\end{array} \right]$ to form ${\bf q}_{k+1}.$ Set $Q_{k+1}=\left[ \begin{array}{c|c}
\begin{array}{c}
Q_k \\
0
\end{array} & {\bf q}_{k+1}
\end{array}\right]$. }
  \item  Compute $\mathcal{V}_{k+1}^{new}=\mathcal{V}_{m+1}(Q_{k+1}\otimes I_s)$ and $\bar{H}_k^{new}=Q_{k+1}^T\bar{H}_mQ_k$, note that $\bar{H}_k^{new}$ is generally a full matrix. Update $D$ according to Options 1--3.
  \item Run the weighted global Arnoldi algorithm from index $(k+1)$ to $m$ to obtain $\mathcal{V}_{m+1}$ and $\bar{H}_m$, where the $(k+1)$th block is the last $s$ columns of ${\mathcal{V}}_{k+1}^{new}$.
  \item Set $X_0=X_m, R_0=R_m$ and ${\bf c}=\mathcal{V}_{m+1}^{T}\diamond_{D} R_0$, and go to Step 6.
\end{enumerate}
\end{algo}
\medskip
\section{Numerical Experiments}\label{sec4}
In this section, we perform some numerical experiments to show the potential of our new algorithms for solving large Sylvester matrix equations.
All the numerical examples were performed using MATLAB R2013b on PC-Pentium(R) with CPU 2.66 GHz and 4.00 of RAM.
In all the algorithms, we set $X_{0}=0_{n\times s}$ to be the initial guess, and choose $C=sprand(n,s,s)$ as the right-hand side matrix, where $sprand(n,s,s)$ is the MATLAB command generating a random, $n$-by-$s$, sparse matrix with approximately $s\times n\times s$ uniformly distributed nonzero entries. Moreover, we use the stopping criterion
\begin{equation}\label{411}
\frac{\Vert\mathcal{V}_{m+1}((\beta {\bf e}_1-\bar{H}_m{\bf y}^{w}_m)\otimes I_s)\Vert_D}{\Vert C\Vert_D}\leq tol=10^{-6},
\end{equation}
and all the algorithms will be stopped as soon as the maximal iteration number $maxit=2500$ is reached. For all the algorithms, we consider comparisons in three aspects: the number of iterations (referred to  {\tt iter}), the runtime in terms of seconds (referred to {\tt CPU}) and the ``real" residual in terms of Frobenius norm (referred to {\tt res.norm})
defined as
\begin{equation}\label{422}
res.norm=\frac{\Vert C-AX_m-X_mB\Vert_F}{\Vert C\Vert_F},
\end{equation}
where $X_m$ are the computed solutions from the algorithms.
In the tables below, we denote by $m,k$ the number of global Arnoldi process and the number of harmonic Ritz block vectors added to the search subspace, respectively.

\medskip
{\bf Example 1.}~{In this example, we illustrate the numerical behavior of Algorithm \ref{alg2} (W-GLGMRES) for different choices of $D$, and show efficiency of our three new weighting strategies proposed in Options 1--3.
To this aim, we compare the performance of W-GLGMRES with the standard global GMRES algorithm (GLGMRES) proposed in \cite{Panjehali}.

The matrices $A$ and $B$ are obtained from the discretization of the operators \cite{Agoujil}
$$
L_{i}(u)=\Delta u-f_{1,i}(x,y)\frac{\partial u}{\partial x}-f_{2,i}(x,y)\frac{\partial u}{\partial y}-f_{3,i}(x,y)u,\quad i=1,2,
$$
on the unit square $[0,1]\times [0,1]$ with homogeneous Dirichlet boundary conditions. In these operators, we have
$f_{1,1}(x,y)=e^{x^2+y}, f_{1,2}(x,y)=2xy, f_{2,1}(x,y)=\sin(x+2y), f_{2,2}(x,y)=e^{xy}$, $ f_{3,1}(x,y)=\cos(xy)$ and $f_{3,2}(x,y)=xy.$ The dimensions of matrices $A$ and $B$ are $n=n_{0}^2$ and $s=s_{0}^2$, respectively. By using the command $fdm\_2d\_matrix$ from LYAPACK \cite{Penzel}, we can extract the matrices $A=fdm(f_{1,1}, f_{2,1}, f_{3,1})$ and $B=fdm(f_{1,2}, f_{2,2}, f_{3,2})$.

We make use of three cases for $D$, i.e., $D_1,D_2$ and $D_3$, which are proposed in Options 1--3. Note that they could be updated during the cycles.
We also consider the case of $D=I_{n\times n}$ in which Algorithm \ref{alg2} reduces to the standard GLGMRES algorithm for large Sylvester matrix equations \cite{Panjehali}.
Table \ref{tab1} lists the numerical results for different choices of $m,s$ and $n$; and Figure \ref{figure1} plots the convergence curves of the algorithms for $n=22500$ and 40000 as $s=16, m=15$.

From Table \ref{tab1} and Figure \ref{figure1}, we observe that the three weighted GLGMRES algorithms need much fewer iterations and much less CPU time than the standard GLGMRES algorithm, and they reach about the same accuracy in terms of the ``real" residual norm. More precisely, W-GLGMRES performs better than the standard GLGMRES algorithm, using $D_1$, $D_2$ or $D_3$ as the weighting matrix; and it works the best with $D_3$.
All these demonstrate that the WGLGMRES algorithm has the potential to improve the convergence, and also it is more robust and efficient than the standard global GMRES algorithm.

\begin{table}
  \centering
  \caption {Example 1: Performance of W-GLGMRES with different choices of $D$}\label{tab1}
  \begin{tabular}{c|c|c|c|c|c|c|c|c}
& & & \multicolumn{3}{c|}{\scriptsize{$n=22500$}} & \multicolumn{3}{|c}{\scriptsize{$n=40000$}} \\
\cline{4-9}
\scriptsize{$s$}&\scriptsize{$m$}&\scriptsize{$D$} &\scriptsize{iter} & \scriptsize{res.norm} &\scriptsize{CPU}&\scriptsize{iter} & \scriptsize{res.norm} & \scriptsize{CPU} \\
\hline
\scriptsize{}&\scriptsize{}&\scriptsize{$I$} &\scriptsize{271} &\scriptsize{9.6980e-07} & \scriptsize{1.7603e+03}&\scriptsize{502} &\scriptsize{9.8712e-07} & \scriptsize{3.6526e+03}\\
\scriptsize{}&\scriptsize{}&\scriptsize{$D_1$} & \scriptsize{225} & \scriptsize{8.9007e-07} & \scriptsize{1.4532e+03} &\scriptsize{361} &\scriptsize{9.0881e-07} & \scriptsize{2.6656e+03} \\
\scriptsize{25}&\scriptsize{10}&\scriptsize{$D_2$} & \scriptsize{196} & \scriptsize{8.9191e-07}& \scriptsize{1.2564e+03}&\scriptsize{375} & \scriptsize{9.9170e-07}& \scriptsize{2.7972e+03}\\
\scriptsize{}&\scriptsize{}&\scriptsize{$D_3$} & \scriptsize{149} & \scriptsize{8.9245e-07} & \scriptsize{964.4450}&\scriptsize{247} &\scriptsize{9.2564e-07} & \scriptsize{1.8723e+03}\\
\hline
\scriptsize{}&\scriptsize{}&\scriptsize{$I$} &\scriptsize{287} &\scriptsize{9.6370e-07} &\scriptsize{1.0593e+03}&\scriptsize{511}&\scriptsize{9.9187e-07} & \scriptsize{ 2.6007e+03}\\
\scriptsize{}&\scriptsize{}&\scriptsize{$D_1$} & \scriptsize{224}&\scriptsize{8.6859e-07} &\scriptsize{ 819.1613}&\scriptsize{350} &\scriptsize{8.0319e-07} & \scriptsize{1.7281e+03} \\
\scriptsize{16}&\scriptsize{10}&\scriptsize{$D_2$} & \scriptsize{221}& \scriptsize{9.5079e-07}  & \scriptsize{828.4745} &\scriptsize{278} &\scriptsize{9.0354e-07}& \scriptsize{1.2168e+03}\\
\scriptsize{}&\scriptsize{}&\scriptsize{$D_3$} & \scriptsize{147}&\scriptsize{8.9043e-07} & \scriptsize{540.1904}&\scriptsize{253} &\scriptsize{9.1782e-07} & \scriptsize{1.1863e+03}\\
\hline
\scriptsize{}&\scriptsize{}&\scriptsize{$I$} &\scriptsize{135}&\scriptsize{9.9273e-07}  & \scriptsize{1.0023e+03}&\scriptsize{229} &\scriptsize{9.4550e-07} & \scriptsize{2.5741e+03}\\
\scriptsize{}&\scriptsize{}&\scriptsize{$D_1$} & \scriptsize{93}& \scriptsize{9.2973e-07}& \scriptsize{690.4734}&\scriptsize{164}&\scriptsize{9.6980e-07} & \scriptsize{1.8403e+03} \\
\scriptsize{16}&\scriptsize{15}&\scriptsize{$D_2$} & \scriptsize{85}& \scriptsize{8.8796e-07}& \scriptsize{612.2571}&\scriptsize{138}&\scriptsize{9.1576e-07} & \scriptsize{1.5791e+03}\\
\scriptsize{}&\scriptsize{}&\scriptsize{$D_3$} & \scriptsize{77}& \scriptsize{8.9375e-07}& \scriptsize{558.9338}&\scriptsize{125} &\scriptsize{8.8805e-07}& \scriptsize{1.4288e+03}\\
\end{tabular}
\end{table}

\begin{figure}
\centering
 \includegraphics[width=5cm]{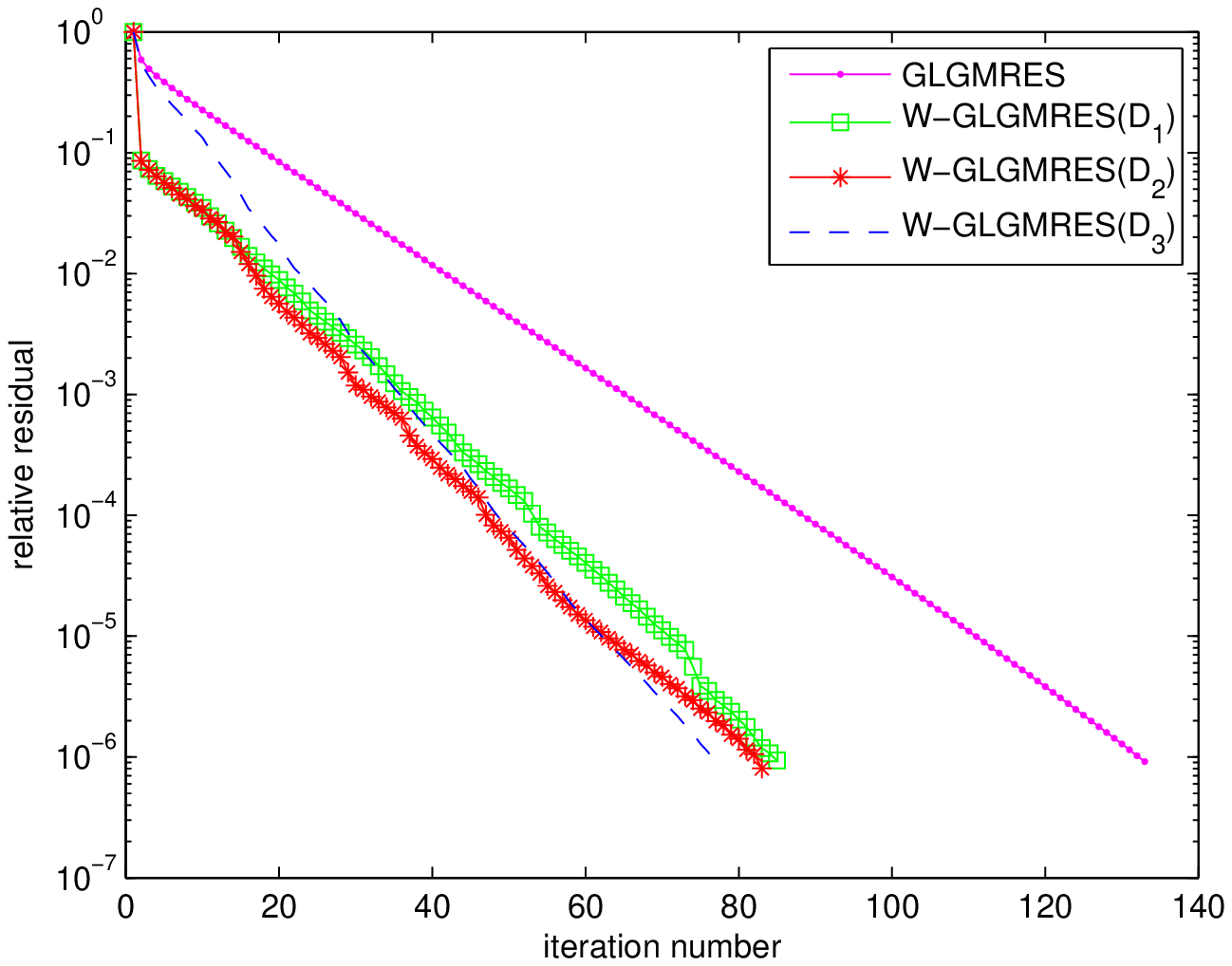}
 \quad
 \includegraphics[width=5cm]{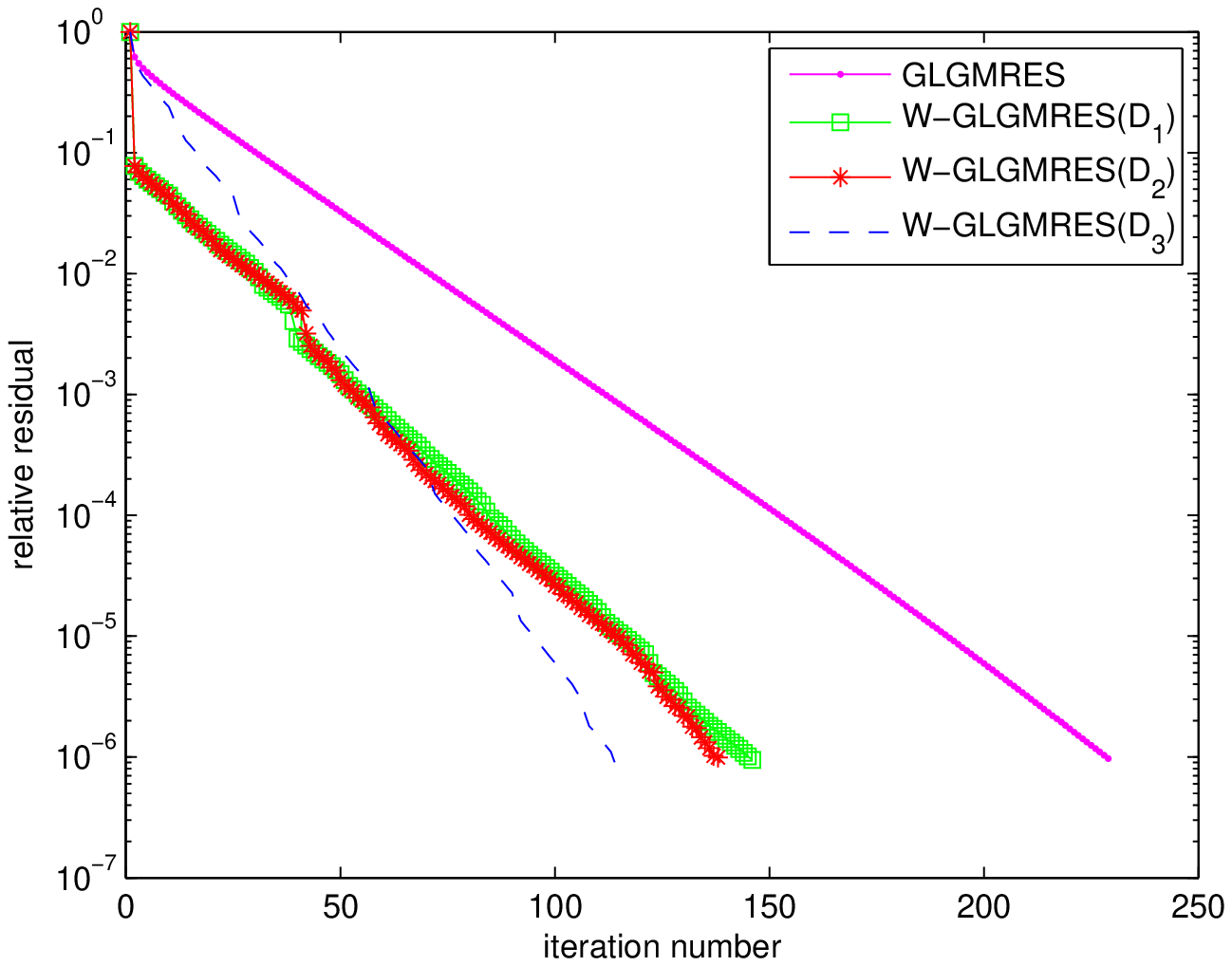}
 \caption{Example 1: Convergence curves of the W-GLGMRES algorithm and those of the standard GLGMRES algorithm. Left: $n=22500, s=16, m=15$; Right: $n=40000, s=16, m=15$. }\label{figure1}
 \end{figure}
\medskip
{\bf Example 2.}~{ In this example, we compare our weighting strategies with the ones proposed in \cite{Heyouniessai,Saberinajafi}, and show effectiveness of our new strategies. In \cite{Heyouniessai}, Heyouni {\it et al.} considered the linear equation $AX=C$,
and proposed a weighted matrix $D$ with elements $(D)_{i,j}=\sqrt{ns}\vert C\vert/\Vert C\Vert_F$, where $\vert C\vert$ is the matrix with absolute values of $C$. They then introduced a weighted strategy as
$(X, Y )_D = tr(X^T(D\circ Y))$, where $D\circ Y$ denotes the Hadamard product of $D$ and $Y$.
In \cite{Saberinajafi}, Saberi Najafi {\it et al.} suggested choosing the weighting matrix $D$ as a diagonal random matrix whose diagonal elements are uniformly and randomly chosen from $(0,2)$. In all the numerical examples from now on, we use $D_3$ as the weighting strategy in our new algorithms.

The test matrices are available from the Matrix Market Collection \cite{Matrix Market website}. The names of these matrices, the size, the density of nonzeros elements and the type of the matrices are shown in the Table \ref{tab2}.
Table \ref{tab3} lists the iteration numbers, CPU time and residual norms of the approximations, obtained from running W-GLGMRES with three different weighting strategies. The results demonstrate that by using our weighting strategy $D_3$, the weighted global GMRES converges faster, and it needs fewer number of iterations and less CPU time than the other two strategies given in \cite{Heyouniessai,Saberinajafi}. In this example, the ``Hadamard product" strategy \cite{Heyouniessai} is better than the ``randomized" strategy \cite{Saberinajafi} according to iteration numbers and CPU time, while our new strategy based on the residual works the best. However, we find that the ``real residual" norm {\it res.norm} computed from the ``Hadamard product" strategy \cite{Heyouniessai} may be larger than the desired tolerance $tol=10^{-6}$ in some cases, and it is obvious to see that our new strategy can cure this drawback very well. Indeed, the stopping criterion used is (\ref{411}) in practical calculations, rather than (\ref{422}).
Figures \ref{figure2} and \ref{figure3} plot the convergence curves of the three algorithms. Again, they illustrate the effectiveness and efficiency of our new choice of weighting matrix.

\begin{table}
\centering
\caption {Summary of of the test matrices for Example 2--Example 4.}\label{tab2}
\begin{tabular}{c|c|c|c|c|c}
\scriptsize{Matrix}&\scriptsize{n}&\scriptsize{nnz} &\scriptsize{Density} & \scriptsize{Density}& \scriptsize{Application area}\\
\cline{1-6}
\scriptsize{saylr4}& \scriptsize{3564 }&\scriptsize{22316 }&\scriptsize{0.0017}&\scriptsize{real unsymmetric} & \scriptsize{Oil reservoir modeling }\\
\scriptsize{add32}& \scriptsize{4960}&\scriptsize{19848}&\scriptsize{0.0008}&\scriptsize{real unsymmetric} & \scriptsize{Electronic circuit design}\\
\scriptsize{sherman4}& \scriptsize{1104}&\scriptsize{3786}&\scriptsize{0.0031}&\scriptsize{real unsymmetric} & \scriptsize{Oil reservior modeling}\\
\scriptsize{sherman2}&\scriptsize{1080}&\scriptsize{23094}&\scriptsize{0.0198}&\scriptsize{real unsymmetric}& \scriptsize{Oil reservoir modeling}\\
\scriptsize{pde2961}&\scriptsize{2961}&\scriptsize{14585}&\scriptsize{0.0231}&\scriptsize{real unsymmetric}& \scriptsize{Partial differential equation }
\end{tabular}
 \end{table}

{\small \begin{table}
\caption{Example 2: Numerical results of W-GLGMRES with different weighting stratiges.
}
\label{tab3}
 \begin{tabular}{c|c|c c c|c c c}
& & \multicolumn{3}{|c|}{ \scriptsize{$m=10$}} & \multicolumn{3}{|c}{\scriptsize{$m=20$}} \\
\cline{3-8}
\scriptsize{Problem} & \scriptsize{Weighting Strategy}&  \scriptsize{iter} & \scriptsize{res.norm} & \scriptsize{CPU} &  \scriptsize{iter} & \scriptsize{res.norm} & \scriptsize{CPU} \\
\hline
\scriptsize{$A=sherman4$} &\scriptsize\cite{Heyouniessai}& \scriptsize{102} &\scriptsize{1.8165e-06}& \scriptsize{997.0492}& \scriptsize{39}& \scriptsize{1.7936e-06} & \scriptsize{1.0142e+03}\\
\scriptsize{$B=fdm(cos(xy),e^{y^2x},100)$} &\scriptsize\cite{Saberinajafi}& \scriptsize{143} &\scriptsize{7.9507e-07}& \scriptsize{1.4222e+03}& \scriptsize{44}& \scriptsize{9.6308e-07} & \scriptsize{1.8846e+03}\\
\scriptsize{$n=1104, s=400$} &\scriptsize{$D_3$}& \scriptsize{38} &\scriptsize{6.1240e-07}& \scriptsize{371.8751}& \scriptsize{16}& \scriptsize{8.9240e-07} & \scriptsize{374.3088}\\
\hline
\scriptsize{$A=add32$} &\scriptsize\cite{Heyouniessai}& \scriptsize{25} &\scriptsize{8.0562e-06}& \scriptsize{1.6571e+03}& \scriptsize{12}& \scriptsize{3.4655e-06} & \scriptsize{1.3972e+03}\\
\scriptsize{$B=fdm(sin(xy),e^{xy},10)$} &\scriptsize\cite{Saberinajafi}& \scriptsize{40} &\scriptsize{8.5137e-07}& \scriptsize{2.2993e+03}& \scriptsize{24}& \scriptsize{8.0102e-07} & \scriptsize{2.7319e+03}\\
\scriptsize{$n=4960, s=400$} &\scriptsize{$D_3$}& \scriptsize{26} &\scriptsize{8.1566e-07}& \scriptsize{957.1628}& \scriptsize{10}& \scriptsize{6.9810e-07} & \scriptsize{934.2591}\\
\hline
\scriptsize{$A=saylr4$} &\scriptsize\cite{Heyouniessai}& \scriptsize{84} &\scriptsize{1.4897e-06}& \scriptsize{3.8107e+03}& \scriptsize{34}& \scriptsize{1.2236e-06} & \scriptsize{2.9038e+03}\\
\scriptsize{$B=fdm(sin(xy),e^{xy},10)$} &\scriptsize\cite{Saberinajafi}& \scriptsize{115} &\scriptsize{9.2482e-07}& \scriptsize{4.9978e+03}& \scriptsize{40}& \scriptsize{6.45505e-07} & \scriptsize{4.7578e+03}\\
\scriptsize{$n=3564, s=400$} &\scriptsize{$D_3$}& \scriptsize{42} &\scriptsize{9.0481e-07}& \scriptsize{1.7269e+03}& \scriptsize{16}& \scriptsize{7.1080e-07} & \scriptsize{1.8709e+03}\\
\hline
\scriptsize{$A=sherman2$} &\scriptsize\cite{Heyouniessai}& \scriptsize{31} &\scriptsize{ 7.7570e-07}& \scriptsize{996.7216}& \scriptsize{13}& \scriptsize{7.4740e-07} & \scriptsize{843.9966}\\
\scriptsize{$B=fdm(sin(xy),e^{xy},10)$} &\scriptsize\cite{Saberinajafi}& \scriptsize{48} &\scriptsize{8.0974e-07}& \scriptsize{1.8602e+03}& \scriptsize{16}& \scriptsize{6.9319e-07} & \scriptsize{2.5471e+03}\\
\scriptsize{$n=1080, s=400$} &\scriptsize{$D_3$}& \scriptsize{20} &\scriptsize{ 1.9301e-07}& \scriptsize{544.5839}& \scriptsize{12}& \scriptsize{2.0916e-07} & \scriptsize{668.4175}\\
\hline
\scriptsize{$A=pde2961$} &\scriptsize\cite{Heyouniessai}& \scriptsize{32} &\scriptsize{1.3552e-06}&\scriptsize{1.0066e+03} & \scriptsize{12}& \scriptsize{1.4062e-06} & \scriptsize{879.7433}\\
\scriptsize{$B=fdm(sin(xy),e^{xy},10)$} &\scriptsize\cite{Saberinajafi}& \scriptsize{47} &\scriptsize{8.5788e-07}& \scriptsize{1.5451e+03}& \scriptsize{16}& \scriptsize{8.6097e-07} & \scriptsize{1.1724e+03}\\
\scriptsize{$n=2961, s=400$} &\scriptsize{$D_3$}& \scriptsize{28} &\scriptsize{6.6487e-07}& \scriptsize{878.394}& \scriptsize{11}& \scriptsize{9.8081e-07} & \scriptsize{795.3090}\\
\end{tabular}
\end{table}}

\begin{figure}
\centering
  \includegraphics[width=5cm]{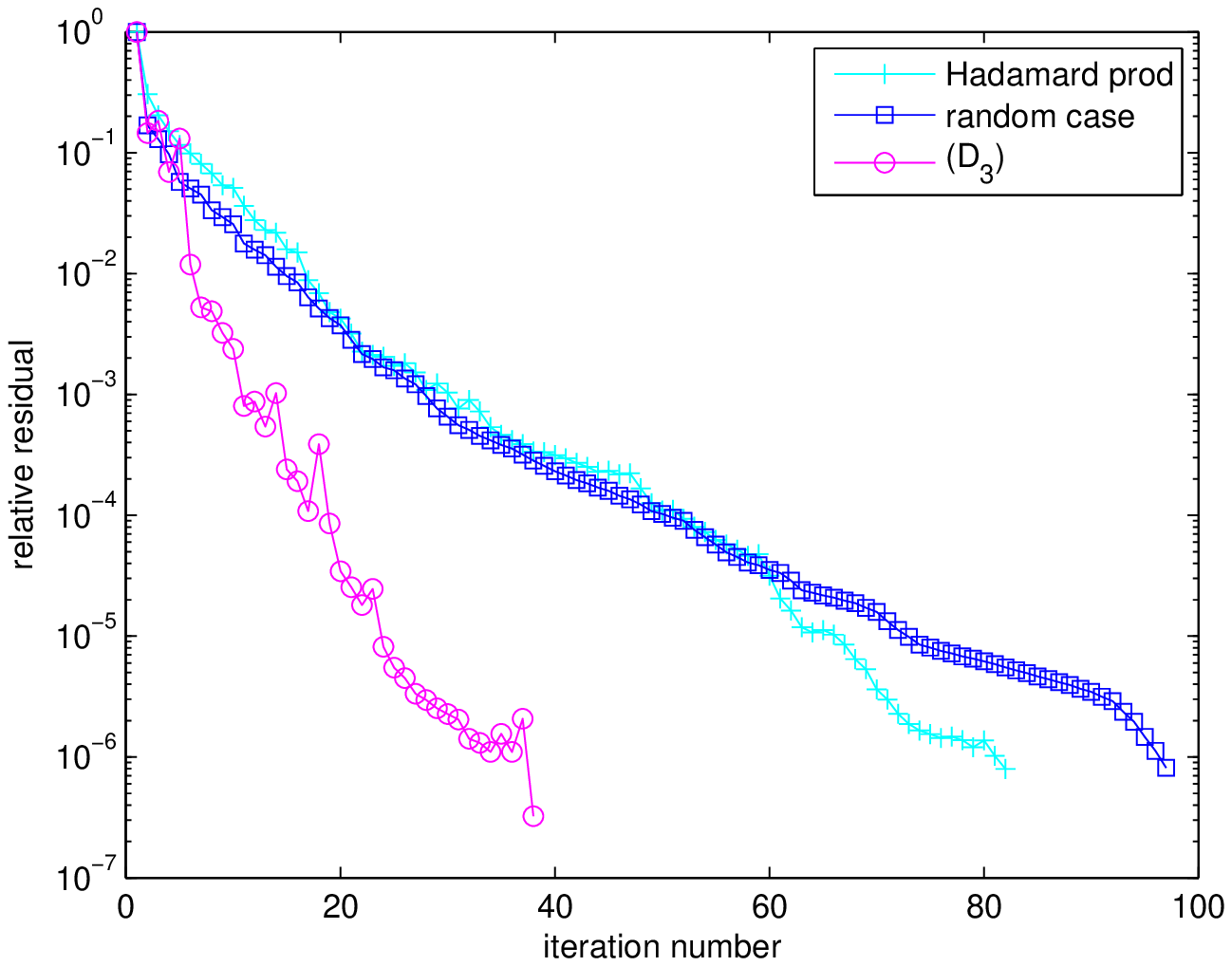}
  \qquad
 \includegraphics[width=5cm]{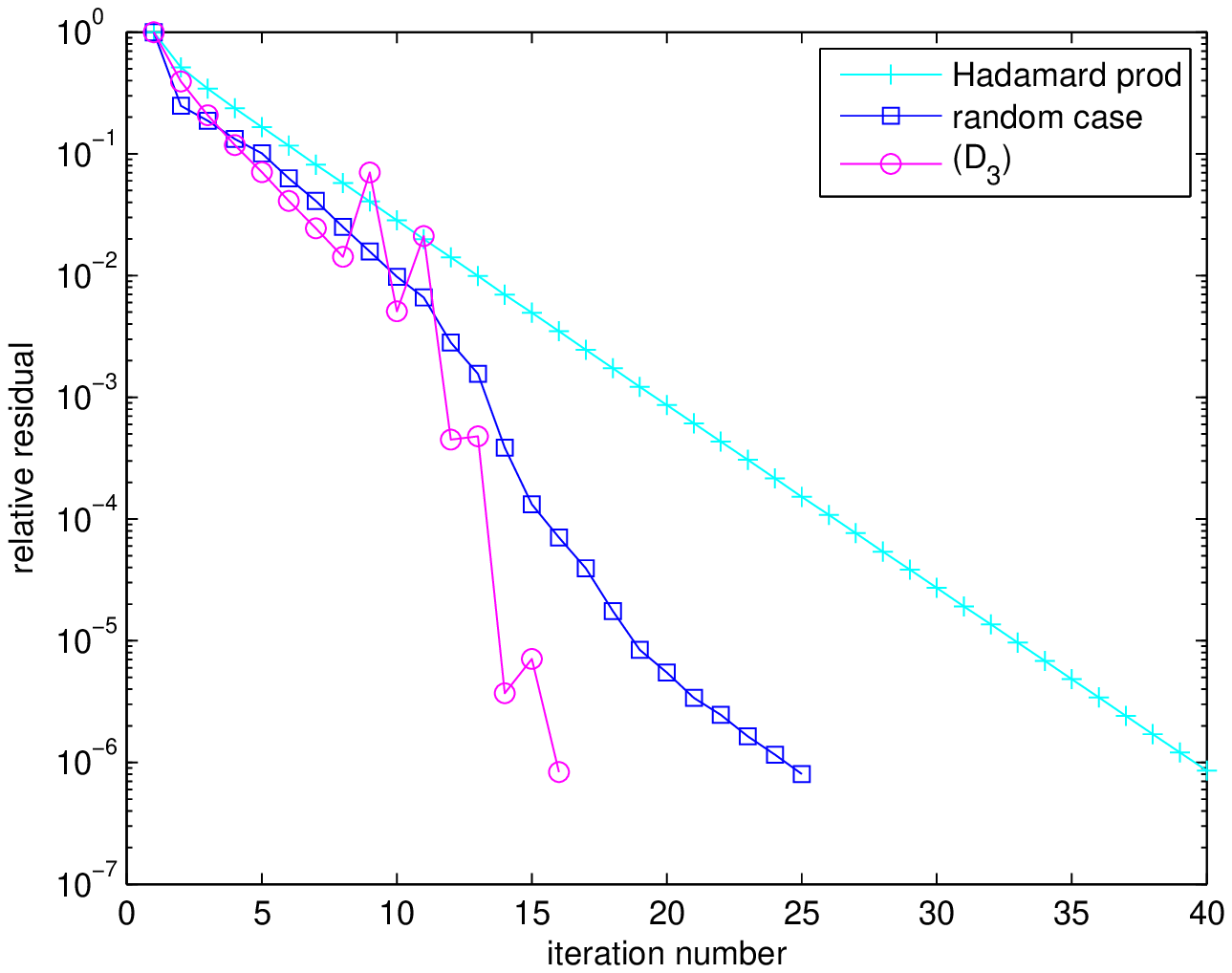}
 \caption{Example 2: Convergence curves of W-GLGMRES with different choices of $D$: the ``Hadamard product" strategy \cite{Heyouniessai}, the ``randomized" strategy \cite{Saberinajafi}, and $D_3$. Left: $A=sherman4$, Right: $A=add32$, $m=10$.  }\label{figure2}
  \end{figure}
\begin{figure}
\centering
\includegraphics[width=5cm]{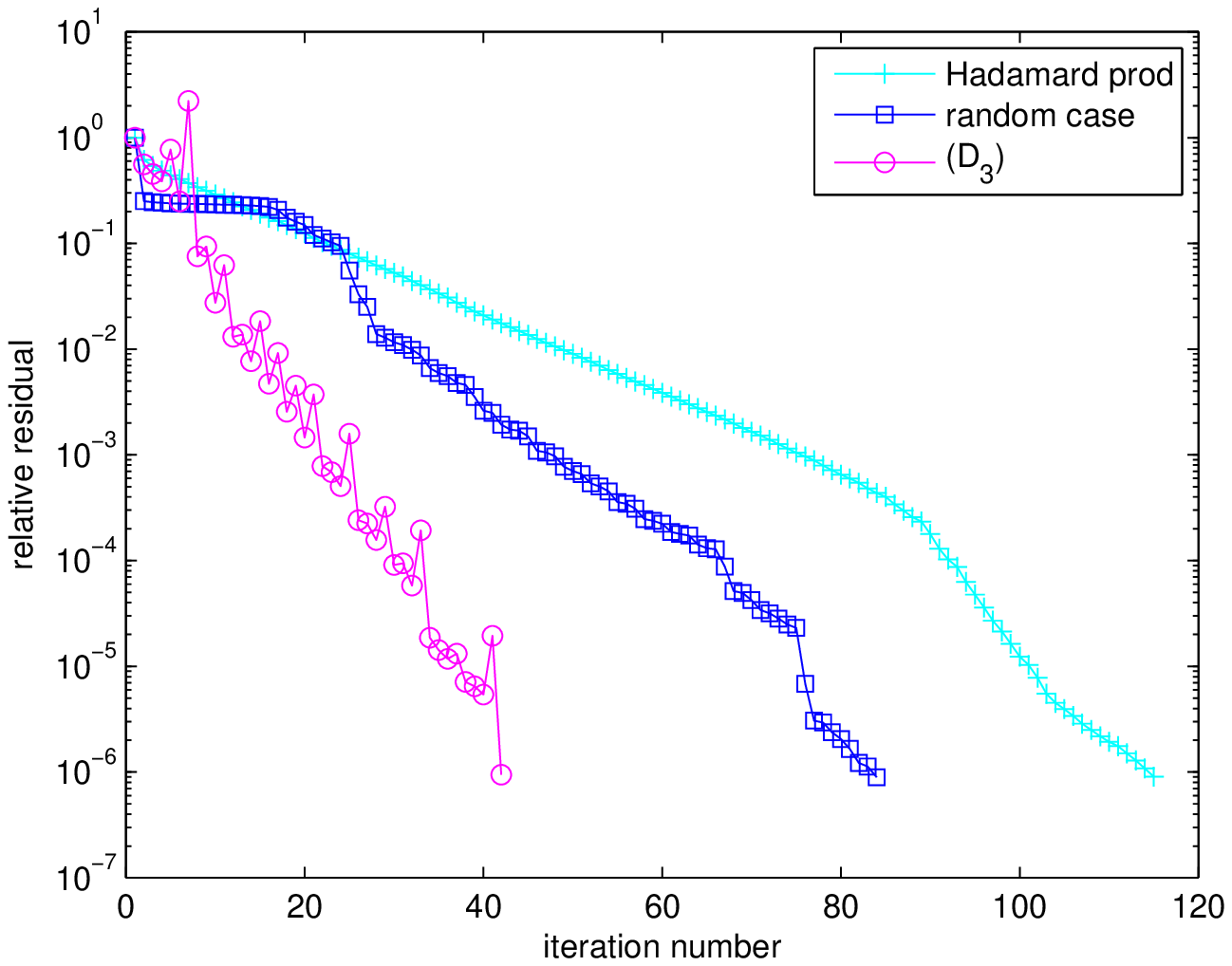}
\quad
\includegraphics[width=5cm]{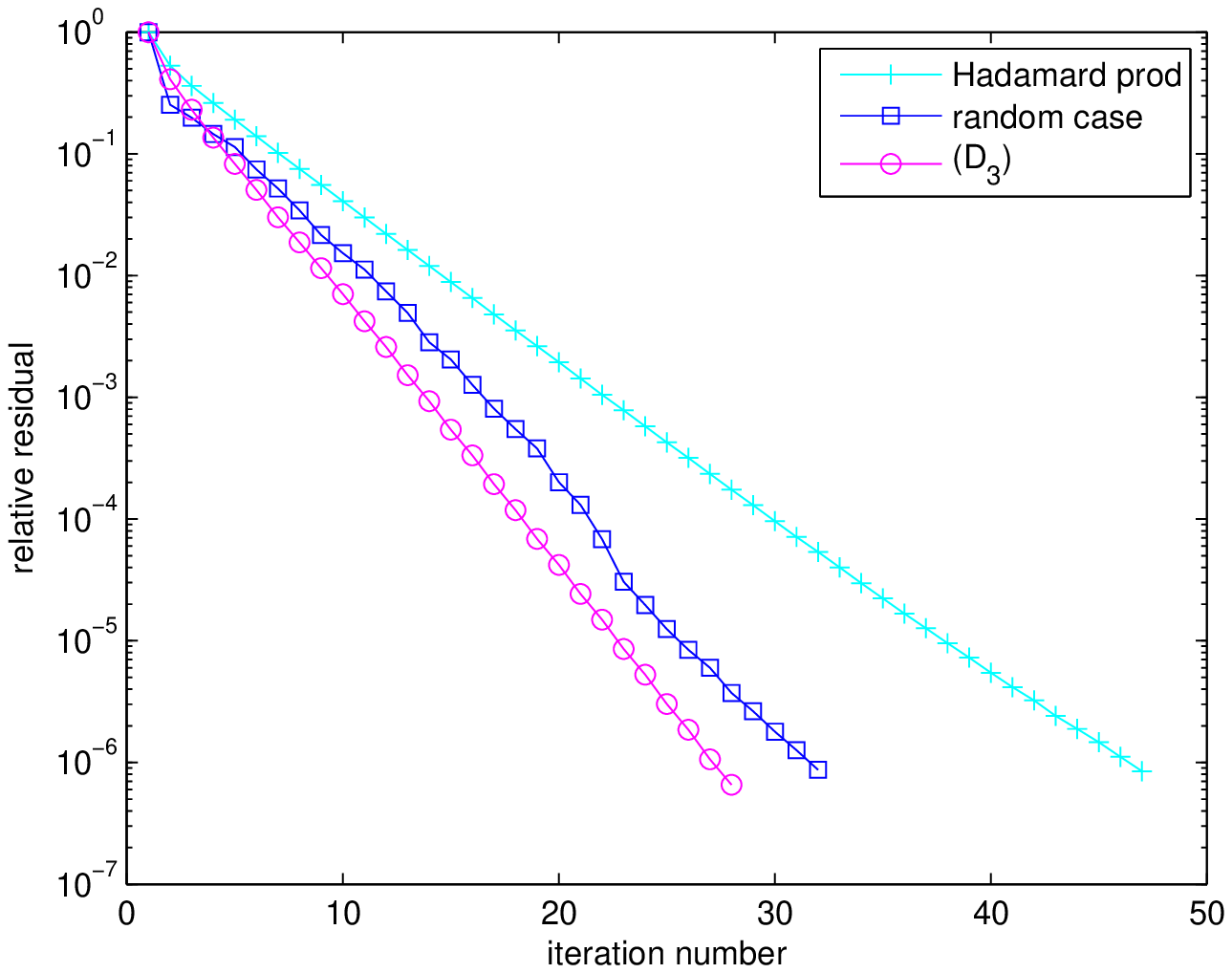}
 \caption{Example 2: Convergence curves of W-GLGMRES with different choices of $D$: the ``Hadamard product" strategy \cite{Heyouniessai}, the ``randomized" strategy \cite{Saberinajafi}, and $D_3$. Left: $A=saylr4$, Right: $A=pde2961$, $m=10$.}\label{figure3}
  \end{figure}

\medskip
{\bf Example 3.}~{When $D=I$, Algorithm \ref{alg3} reduces to the global GMRES algorithm with deflation, which is mathematically equivalent to the algorithm proposed in \cite{Lin}. In this example, we try to show that the weighted global GMRES with deflation is more efficient than the global GMRES algorithm with deflation. To show the efficiency of Algorithm \ref{alg3} (W-GLGMRES-D), we compare it with the global GMRES algorithm (GLGMRES), Algorithm \ref{alg2} (W-GLGMRES), and the global GMRES algorithm with deflation (GLGMRES-D).
In the first test problem, we use $B=fdm(\cos(xy),e^{y^2x},100)$}, and for test problems 2--5, we use $B=fdm(\sin(xy),e^{xy},10)$.

Table \ref{tab4} reports the results of the five test problems, where we use different $m$ and $k$ in W-GLGMRES-D and GLGMRES-D. It is seen that both W-GLGMRES-D and W-GLGMRES outperforms GLGMRES-D and GLGMRES in most cases, which illustrates the effectiveness of our weighting strategy. Furthermore,
W-GLGMRES-D is superior to the other three algorithms in terms of iteration numbers, CPU time, and accuracy.
Specifically, for the 4-th test problem, we see GLGMRES and W-GLGMRES fail to converge with in 2500 iterations, while the algorithms with deflation work quite well. This illustrates that the deflation strategy can improve convergence of the standard global GMRES algorithms for large Sylvester matrix equations.
In addition, in Figures \ref{figure4} and \ref{figure5}, we compare GLGMRES, WGLGMRES, GLGMRES-D and W-GLGMRES-D, where $m=20,k=10$. From Table \ref{tab4} and Figures \ref{figure4} --\ref{figure5}, we conclude that
applying deflation strategy with weighting technique leads to much better solutions.

\begin{table}\footnotesize
 \centering
\caption{Example 3: Numerical results of the four algorithms.}\label{tab4}
\begin{tabular}{c|c|c|c|c|c|c}
\scriptsize{} & \scriptsize{Problem}& \scriptsize{$(m,k)$}&\scriptsize{Algorithm}&\scriptsize{iter} &\scriptsize{res.norm} &  \scriptsize{CPU}\\
\cline{1-7}
\scriptsize{} &\scriptsize{}& \scriptsize{(20,-)}& \scriptsize{GLGMRES}&\scriptsize{60}&\scriptsize{9.2014e-07} &\scriptsize{871.0590}\\
\scriptsize{} &\scriptsize{} & \scriptsize{(20,-)}& \scriptsize{W-GLGMRES} &\scriptsize{18}&\scriptsize{7.8295e-07}& \scriptsize{387.4749}\\
\scriptsize{} & \scriptsize{$A=sherman4$} & \scriptsize{(20,10)}&  \scriptsize{GLGMRES-D}& \scriptsize{7}&\scriptsize{5.8933e-07} &\scriptsize{146.6579}\\
\scriptsize{1} &\scriptsize{} & \scriptsize{(20,10)}& \scriptsize{W-GLGMRES-D}&\scriptsize{4}&\scriptsize{6.1324e-13} & \scriptsize{79.3758}\\
\scriptsize{} &\scriptsize{$n=1104, s=400$} & \scriptsize{(20,15)}& \scriptsize{GLGMRES-D}&\scriptsize{7}&\scriptsize{5.5915e-08} &\scriptsize{103.9108}\\
\scriptsize{} &\scriptsize{}& \scriptsize{(20,15)}& \scriptsize{W-GLGMRES-D}&\scriptsize{5}&\scriptsize{1.7006e-14} &\scriptsize{68.0248}\\
\scriptsize{} &\scriptsize{}& \scriptsize{(30,20)}& \scriptsize{GLGMRES-D}&\scriptsize{8}&\scriptsize{3.0370e-08} & \scriptsize{243.1432}\\
\scriptsize{} &\scriptsize{}& \scriptsize{(30,20)}& \scriptsize{W-GLGMRES-D}&\scriptsize{5}&\scriptsize{5.1292e-14} &\scriptsize{179.8224}\\
 \hline
 \scriptsize{} &\scriptsize{}& \scriptsize{(20,-)}& \scriptsize{GLGMRES}&\scriptsize{14}&\scriptsize{6.7001e-07} &\scriptsize{1.1432e+03}\\
\scriptsize{} &\scriptsize{}& \scriptsize{(20,-)}& \scriptsize{W-GLGMRES} &\scriptsize{11} &\scriptsize{2.9613e-07}& \scriptsize{1.2481e+03}\\
\scriptsize{} &\scriptsize{$A=add32$} & \scriptsize{(20,10)}& \scriptsize{GLGMRES-D}&\scriptsize{7}&  \scriptsize{2.0322e-08} & \scriptsize{581.1866}\\
\scriptsize{2} &\scriptsize{}& \scriptsize{(20,10)} & \scriptsize{W-GLGMRES-D} &\scriptsize{4} &\scriptsize{2.7931e-09}&  \scriptsize{348.0351}\\
\scriptsize{} &\scriptsize{$n=4960, s=400$}& \scriptsize{(20,15)}& \scriptsize{GLGMRES-D}& \scriptsize{7} &\scriptsize{3.5901e-08} & \scriptsize{549.6019}\\
\scriptsize{} &\scriptsize{} & \scriptsize{(20,15)}& \scriptsize{W-GLGMRES-D} &\scriptsize{4} &\scriptsize{1.2383e-09} & \scriptsize{309.2734}\\
\scriptsize{} &\scriptsize{}& \scriptsize{(30,20)}& \scriptsize{GLGMRES-D}&\scriptsize{6}&\scriptsize{4.6590e-08} &\scriptsize{991.0881}\\
\scriptsize{} &\scriptsize{}& \scriptsize{(30,20)}& \scriptsize{W-GLGMRES-D}&\scriptsize{4}&\scriptsize{3.4711e-07} &\scriptsize{700.5107}\\
 \hline
\scriptsize{} &\scriptsize{}& \scriptsize{(20,-)}& \scriptsize{GLGMRES}&\scriptsize{41}&\scriptsize{7.9900e-07} &\scriptsize{1.9829e+03}\\
\scriptsize{} & \scriptsize{} & \scriptsize{(20,-)}& \scriptsize{W-GLGMRES} &\scriptsize{18}&\scriptsize{6.1862e-07}&\scriptsize{1.5938e+03}\\
\scriptsize{} &\scriptsize{$A=saylr4$} & \scriptsize{(20,10)}& \scriptsize{GLGMRES-D}&\scriptsize{7} &\scriptsize{2.1766e-07} &\scriptsize{445.2339}\\
\scriptsize{3} &\scriptsize{}& \scriptsize{(20,10)}& \scriptsize{W-GLGMRES}&\scriptsize{4} &\scriptsize{9.2134e-07}&\scriptsize{257.8798}\\
\scriptsize{} &\scriptsize{$n=3564, s=400$}& \scriptsize{(20,15)}& \scriptsize{GLGMRES-D}&\scriptsize{7} &\scriptsize{6.9319e-10}& \scriptsize{424.5378}\\
\scriptsize{} &\scriptsize{}& \scriptsize{(20,15)}& \scriptsize{W-GLGMRES-D}&\scriptsize{6}&\scriptsize{2.3769e-09}& \scriptsize{ 346.7902}\\
\scriptsize{} &\scriptsize{}& \scriptsize{(30,20)}& \scriptsize{GLGMRES-D}&\scriptsize{7}&\scriptsize{5.1598e-08} &\scriptsize{687.8864}\\
\scriptsize{} &\scriptsize{}& \scriptsize{(30,20)}& \scriptsize{W-GLGMRES-D}&\scriptsize{5}&\scriptsize{9.8619e-07} &\scriptsize{584.5981}\\
 \hline
\scriptsize{} &\scriptsize{}& \scriptsize{(20,-)}& \scriptsize{GLGMRES}&\scriptsize{$\uparrow 2500$}& \scriptsize{$\dagger$}  &\scriptsize{-}\\
\scriptsize{} & \scriptsize{}& \scriptsize{(20,-)}& \scriptsize{W-GLGMRES}&\scriptsize{$\uparrow 2500$} &\scriptsize{$\dagger$} &\scriptsize{-}\\
\scriptsize{} &\scriptsize{$A=sherman2$}& \scriptsize{(20,10)}& \scriptsize{GLGMRES-D}&\scriptsize{5}&\scriptsize{1.8630e-07}  &\scriptsize{101.9935}\\
\scriptsize{4} &\scriptsize{}& \scriptsize{(20,10)}& \scriptsize{W-GLGMRES-D}&\scriptsize{4}& \scriptsize{5.2809e-08}&\scriptsize{96.6050}\\
\scriptsize{} &\scriptsize{$n=1080, s=400$}& \scriptsize{(20,15)}& \scriptsize{GLGMRES-D}&\scriptsize{6}& \scriptsize{8.8742e-09} &\scriptsize{92.2590} \\
\scriptsize{} &\scriptsize{}& \scriptsize{(20,15)}& \scriptsize{W-GLGMRES-D}&\scriptsize{5}&\scriptsize{5.9891e-11} &\scriptsize{ 88.5462}\\
\scriptsize{} &\scriptsize{}& \scriptsize{(30,20)}& \scriptsize{GLGMRES-D}&\scriptsize{6}&\scriptsize{3.1034e-10} &\scriptsize{211.2737}\\
\scriptsize{} &\scriptsize{}& \scriptsize{(30,20)}& \scriptsize{W-GLGMRES-D}&\scriptsize{5}&\scriptsize{8.0189e-12} &\scriptsize{ 210.1942}\\
\hline
\scriptsize{} &\scriptsize{}& \scriptsize{(20,-)}& \scriptsize{GLGMRES}&\scriptsize{16}& \scriptsize{7.8910e-07}  &\scriptsize{865.9652}\\
\scriptsize{} & \scriptsize{}&\scriptsize{(20,-)}&\scriptsize{W-GLGMRES}&\scriptsize{11}&\scriptsize{5.1991e-07}& \scriptsize{606.6930}\\
\scriptsize{} &\scriptsize{$A=pde2961$}& \scriptsize{(20,10)}& \scriptsize{GLGMRES-D}&\scriptsize{7}&\scriptsize{9.6839e-09} & \scriptsize{411.2545}\\
\scriptsize{5} &\scriptsize{}& \scriptsize{(20,10)}& \scriptsize{W-GLGMRES-D}&\scriptsize{4}& \scriptsize{5.0377e-08}  &\scriptsize{238.1900}\\
\scriptsize{} &\scriptsize{$n=2961, s=400$}& \scriptsize{(20,15)}& \scriptsize{GLGMRES-D}&\scriptsize{6}& \scriptsize{ 3.7837e-07}& \scriptsize{348.5320}\\
\scriptsize{} &\scriptsize{}& \scriptsize{(20,15)}& \scriptsize{W-GLGMRES-D}&\scriptsize{5} &\scriptsize{3.7809e-08}&  \scriptsize{289.2056}\\
\scriptsize{} &\scriptsize{}& \scriptsize{(30,20)}& \scriptsize{GLGMRES-D}&\scriptsize{6}&\scriptsize{1.0925e-08} &\scriptsize{576.2694}\\
\scriptsize{} &\scriptsize{}& \scriptsize{(30,20)}& \scriptsize{W-GLGMRES-D}&\scriptsize{4}&\scriptsize{8.1073e-10} &\scriptsize{391.9157}\\
\end{tabular}
 \end{table}

\begin{figure}
\centering
  \includegraphics[width=5cm]{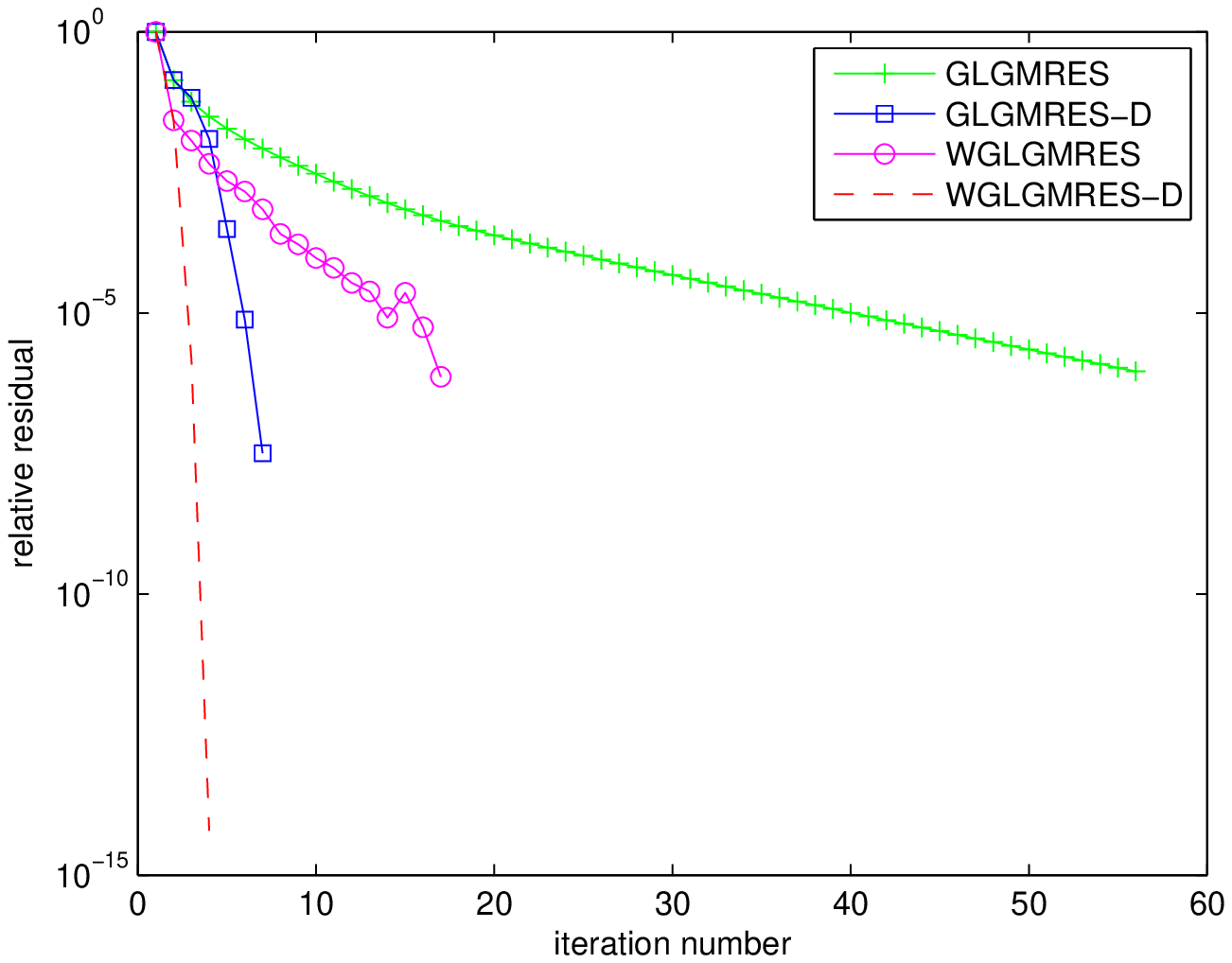}
  \qquad
 \includegraphics[width=5cm]{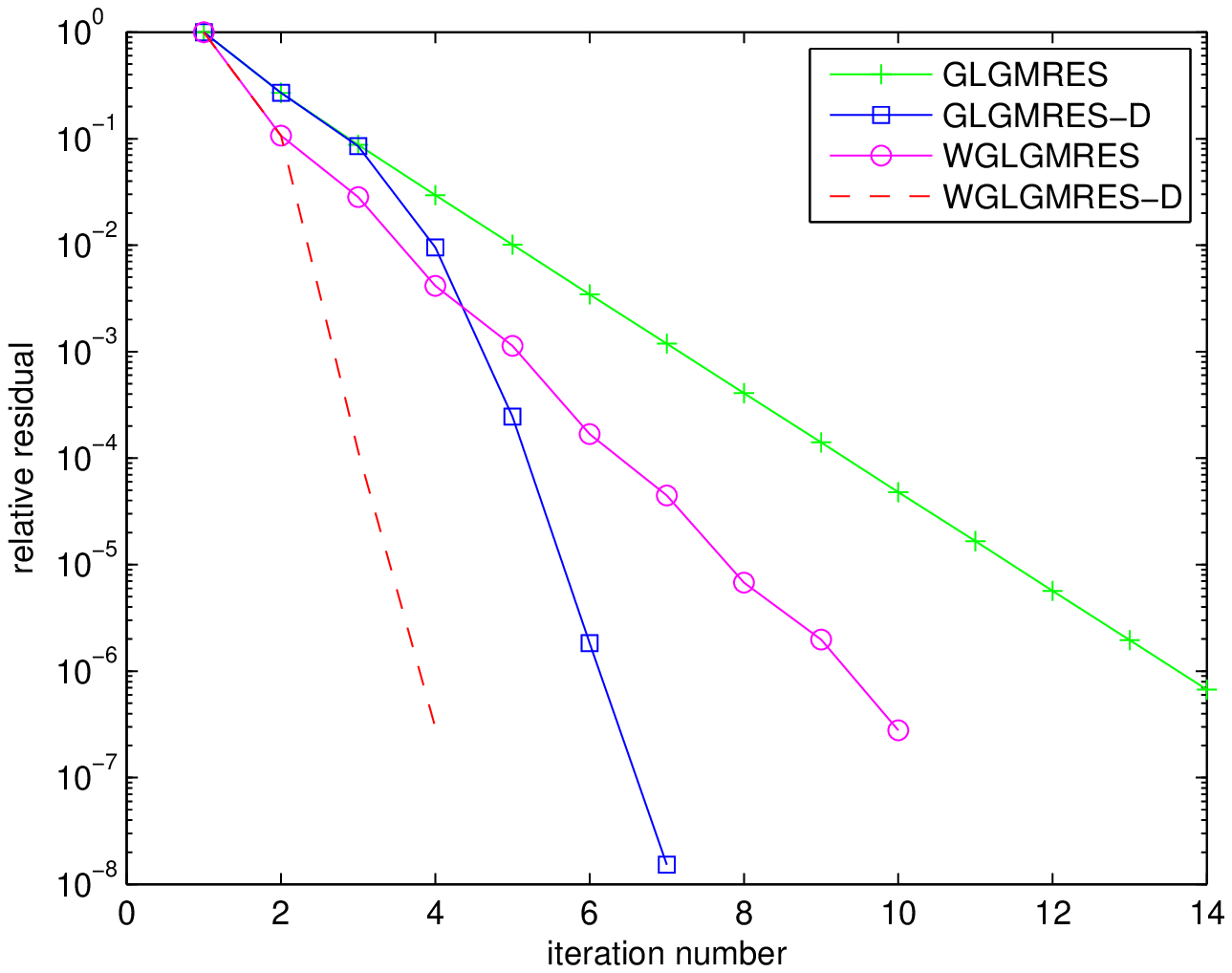}
 \caption{Example 3: Convergence curves of the four algorithms. Left: $A=sherman4$, Right: $A=add32$.  }\label{figure4}
  \end{figure}
\begin{figure}
\centering
\includegraphics[width=4cm]{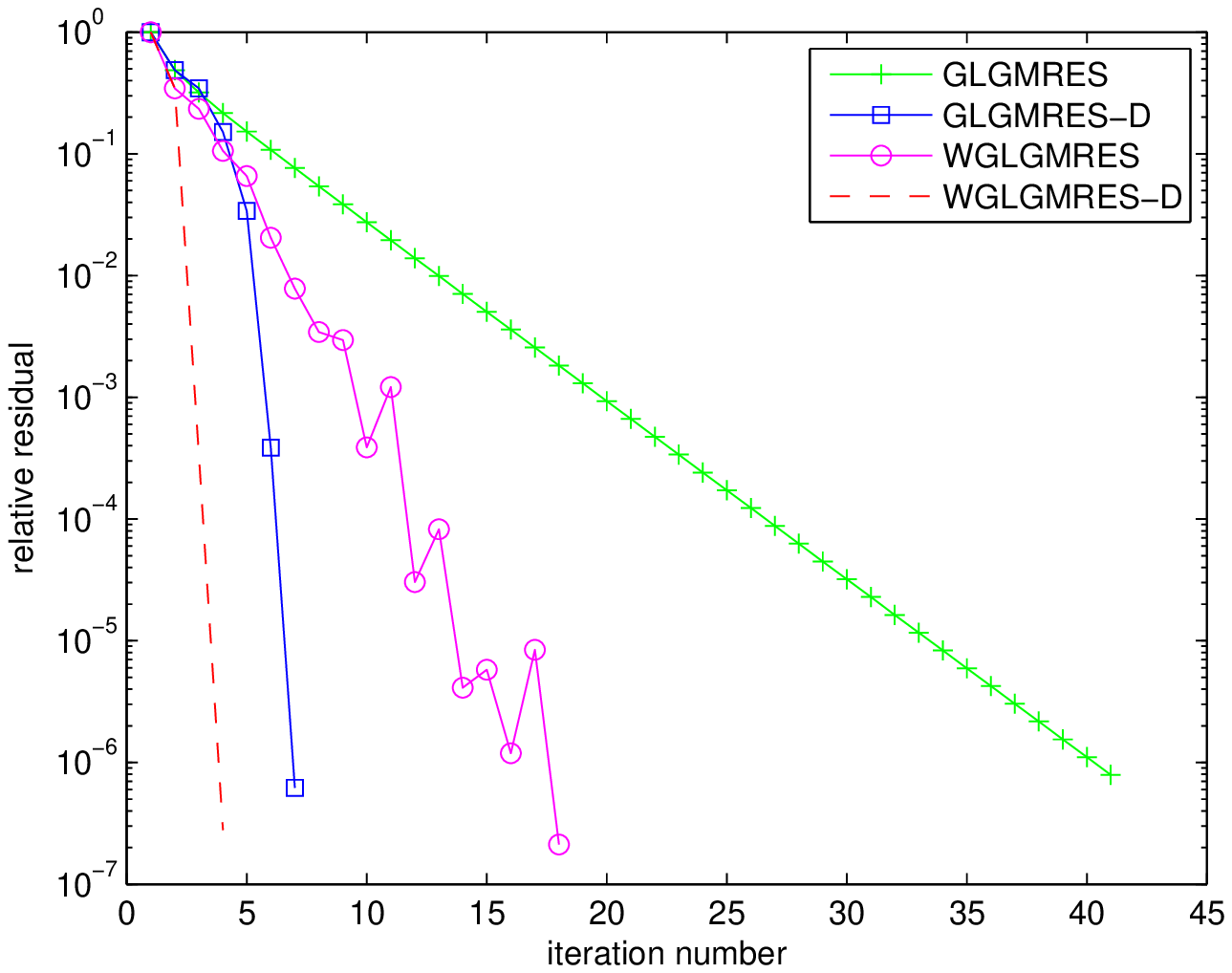}
 \quad
\includegraphics[width=4cm]{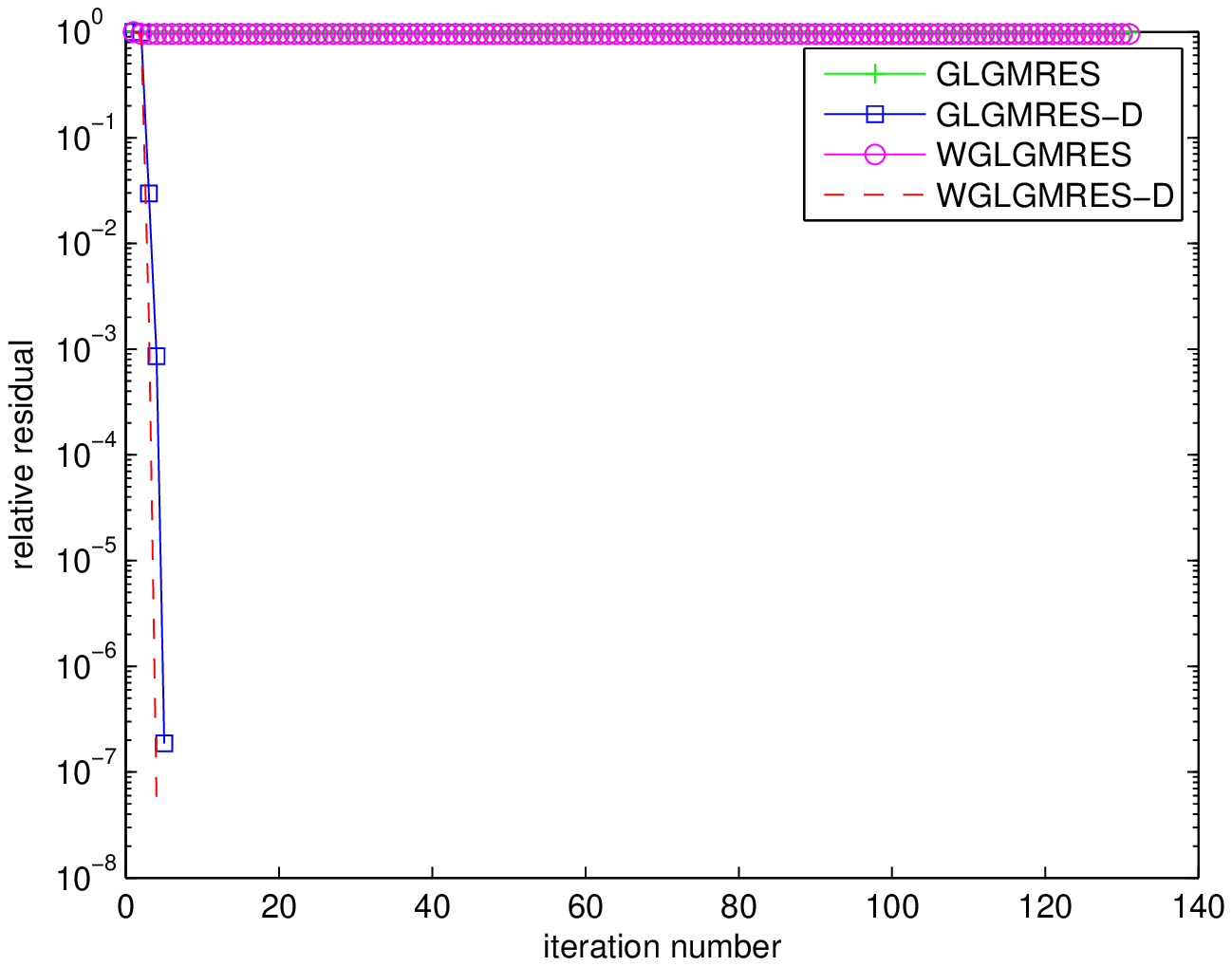}
\quad
\includegraphics[width=4cm]{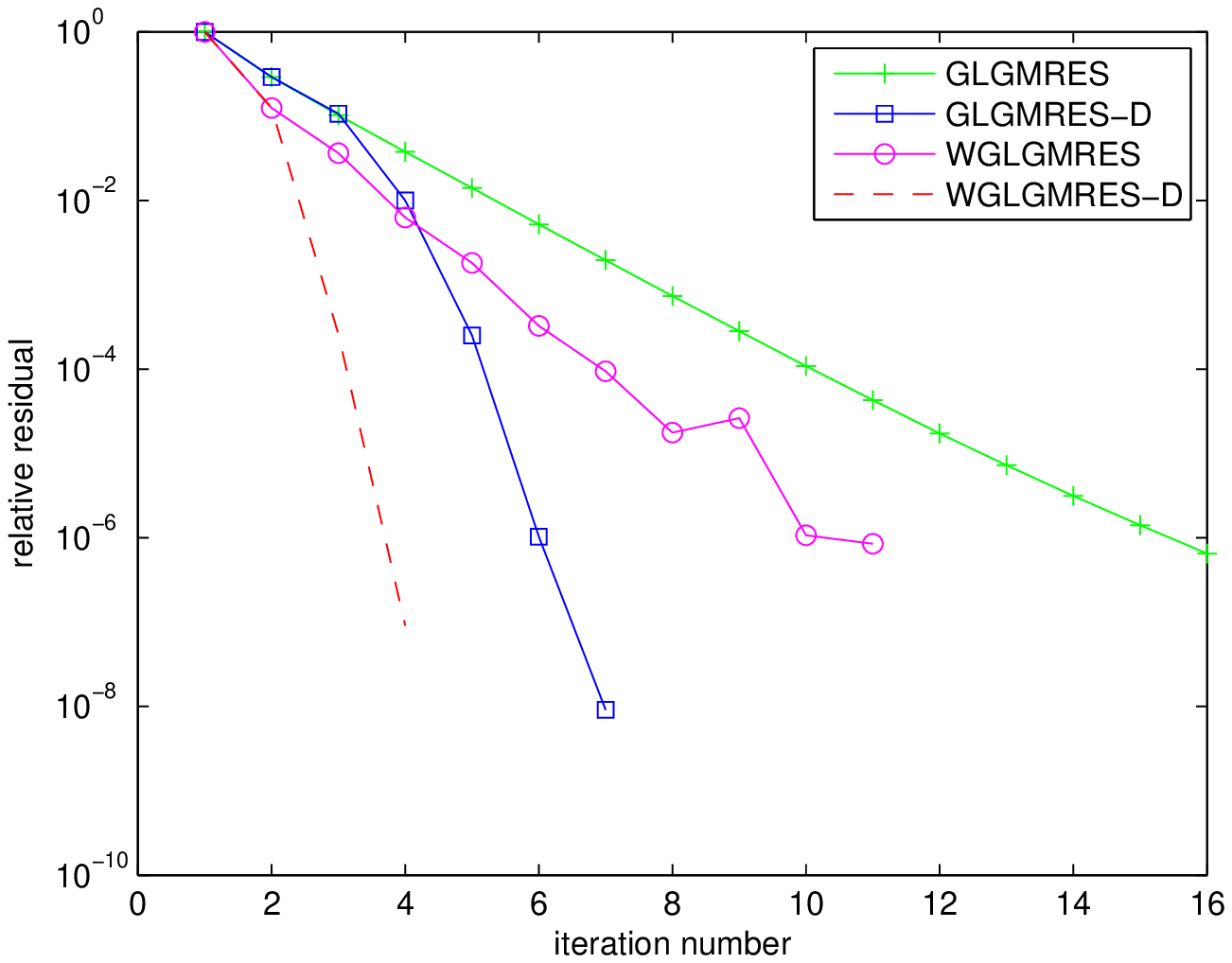}
 \caption{Example 3: Convergence curves of the four algorithms. Left: $A=saylr4$, middle: $A=sherman2$, Right: $A=pde2961$.  }\label{figure5}
  \end{figure}
\medskip
{\bf Example 4. \label{exp4}}~{
In this example, we combine the weighted and deflated strategy with the flexible preconditoning strategy \cite{Saad}, and show the numerical behavior of the resulting algorithm.
In the flexible preconditioned algorithms, the preconditioner may vary from one step to the next, for more details, refer to \cite{Saad}.
In this example, the flexible preconditioner consists of five steps of full GLGMRES for solving the linear systems with multiple right-hand sides in the inner iterations, and we make use of $D_3$ as the weighting strategy in our new algorithm.

We first compare GLGMRES, W-GLGMRES with their flexible preconditioning versions: flexible global GMRES (F-GLGMRES), and weighted flexible global GMRES (WF-GLGMRES).
Table \ref{tab5} reports the number of iterations, CPU time and residual norm of the four algorithms. It is seen that by combining the weighted and flexible strategies together, we pay fewer iterations and less CPU time, compared with the standard global GMRES algorithm and the weighted global GMRES algorithm, except for Sherman4. Indeed, for this problem, both F-GLGMRES and WF-GLGMRES use fewer iterations than W-GLGMRES, while the CPU time used for the two former algorithms is (much) more than W-GLGMRES.
The reason is that in the flexible algorithms, one has to approximately solve $m=20$ linear systems with $s=400$ right-hand sides per cycle. Note that for Sherman2, both the standard and the weighted flexible global GMRES do not work within 2500 iterations.
Figures \ref{figure6}--\ref{figure7} plot the convergence behavior of WF-GLGMRES and F-GLGMRES  for the four problems.

\begin{table}
 \centering
\caption{Example 4: A comparison of flexible GLGMRES and flexible W-GLGMRES.}\label{tab5}
\begin{tabular}{c|c|c c c }
& & \multicolumn{3}{|c}{ \scriptsize{$m=20$}}  \\
\cline{3-5}
\scriptsize{Problem} & \scriptsize{Algorithm} & \scriptsize{iter} & \scriptsize{res.norm}&  \scriptsize{CPU}  \\
\hline
\scriptsize{} &\scriptsize{GLGMRES}&\scriptsize{60}&\scriptsize{9.2014e-07} &\scriptsize{844.0590}\\
\scriptsize{$A=sherman4$} &\scriptsize{W-GLGMRES} &\scriptsize{18}&\scriptsize{7.8295e-07}& \scriptsize{398.4749}\\
\scriptsize{$B=fdm(\cos(xy),e^{y^2x},100)$} &\scriptsize{}& \scriptsize{} &\scriptsize{}& \scriptsize{}\\
\scriptsize{$n=1104,s=400$} &\scriptsize{F-GLGMRES}& \scriptsize{10} & \scriptsize{3.1852e-07} &\scriptsize{798.0550}\\
\scriptsize{} &\scriptsize{WF-GLGMRES}& \scriptsize{6} & \scriptsize{9.4672e-07}&\scriptsize{521.1000}\\
\hline
\scriptsize{} & \scriptsize{GLGMRES}&\scriptsize{14}&\scriptsize{6.7001e-07} &\scriptsize{1.1432e+03}\\
\scriptsize{$A=add32$} & \scriptsize{W-GLGMRES} &\scriptsize{11} &\scriptsize{2.9613e-07}& \scriptsize{1.0521e+03}\\
\scriptsize{$B=fdm(\sin(xy),e^{xy},10)$} &\scriptsize{}& \scriptsize{} &\scriptsize{}& \scriptsize{}\\
\scriptsize{$n=4960,s=400$} &\scriptsize{F-GLGMRES}& \scriptsize{3} & \scriptsize{7.9173e-07}&\scriptsize{812.2611}\\
\scriptsize{} &\scriptsize{WF-GLGMRES}& \scriptsize{2} & \scriptsize{6.8148e-11}&\scriptsize{520.8404}\\
\hline
\scriptsize{} & \scriptsize{GLGMRES}&\scriptsize{41}&\scriptsize{7.9900e-07} &\scriptsize{1.9829e+03}\\
\scriptsize{$A=saylr4$} &  \scriptsize{W-GLGMRES} &\scriptsize{18}&\scriptsize{6.1862e-07}&\scriptsize{1.3438e+03}\\
\scriptsize{$B=fdm(\sin(xy),e^{xy},10)$} &\scriptsize{}& \scriptsize{} &\scriptsize{}& \scriptsize{}\\
\scriptsize{$n=3564,s=400$} &\scriptsize{F-GLGMRES}& \scriptsize{5} & \scriptsize{4.9185e-08}&\scriptsize{1.2897e+03}\\
\scriptsize{} &\scriptsize{WF-GLGMRES}& \scriptsize{4} & \scriptsize{6.0365e-09}&\scriptsize{1.0238e+03}\\
\hline
\scriptsize{} & \scriptsize{GLGMRES}&\scriptsize{$\uparrow 2500$}& \scriptsize{$\dagger$}  &\scriptsize{-}\\
\scriptsize{$A=sherman2$} & \scriptsize{W-GLGMRES}&\scriptsize{$\uparrow 2500$} &\scriptsize{$\dagger$} &\scriptsize{-}\\
\scriptsize{$B=fdm(\sin(xy),e^{xy},10)$} &\scriptsize{}& \scriptsize{} &\scriptsize{}& \scriptsize{}\\
\scriptsize{$n=1080,s=400$} &\scriptsize{F-GLGMRES}& \scriptsize{$\uparrow 2500$}& \scriptsize{$\dagger$} & \scriptsize{-}\\
\scriptsize{} &\scriptsize{WF-GLGMRES}& \scriptsize{$\uparrow 2500$} & \scriptsize{$\dagger$}& \scriptsize{-}\\
\hline
\scriptsize{} & \scriptsize{GLGMRES}&\scriptsize{16}& \scriptsize{7.8910e-07}  &\scriptsize{755.9652}\\
\scriptsize{$A=pde2961$}&\scriptsize{W-GLGMRES}&\scriptsize{11}&\scriptsize{5.1991e-07}& \scriptsize{ 661.4130}\\
\scriptsize{$B=fdm(\sin(xy),e^{xy},10)$} &\scriptsize{}& \scriptsize{} &\scriptsize{}& \scriptsize{}\\
\scriptsize{$n=2961,s=400$} &\scriptsize{F-GLGMRES}& \scriptsize{4} & \scriptsize{7.9803e-08}&\scriptsize{777.3842}\\
\scriptsize{} &\scriptsize{WF-GLGMRES}& \scriptsize{3} & \scriptsize{6.1608e-08}&\scriptsize{573.7325}\\
\end{tabular}
\end{table}

\begin{figure}
\centering

  \includegraphics[width=5cm]{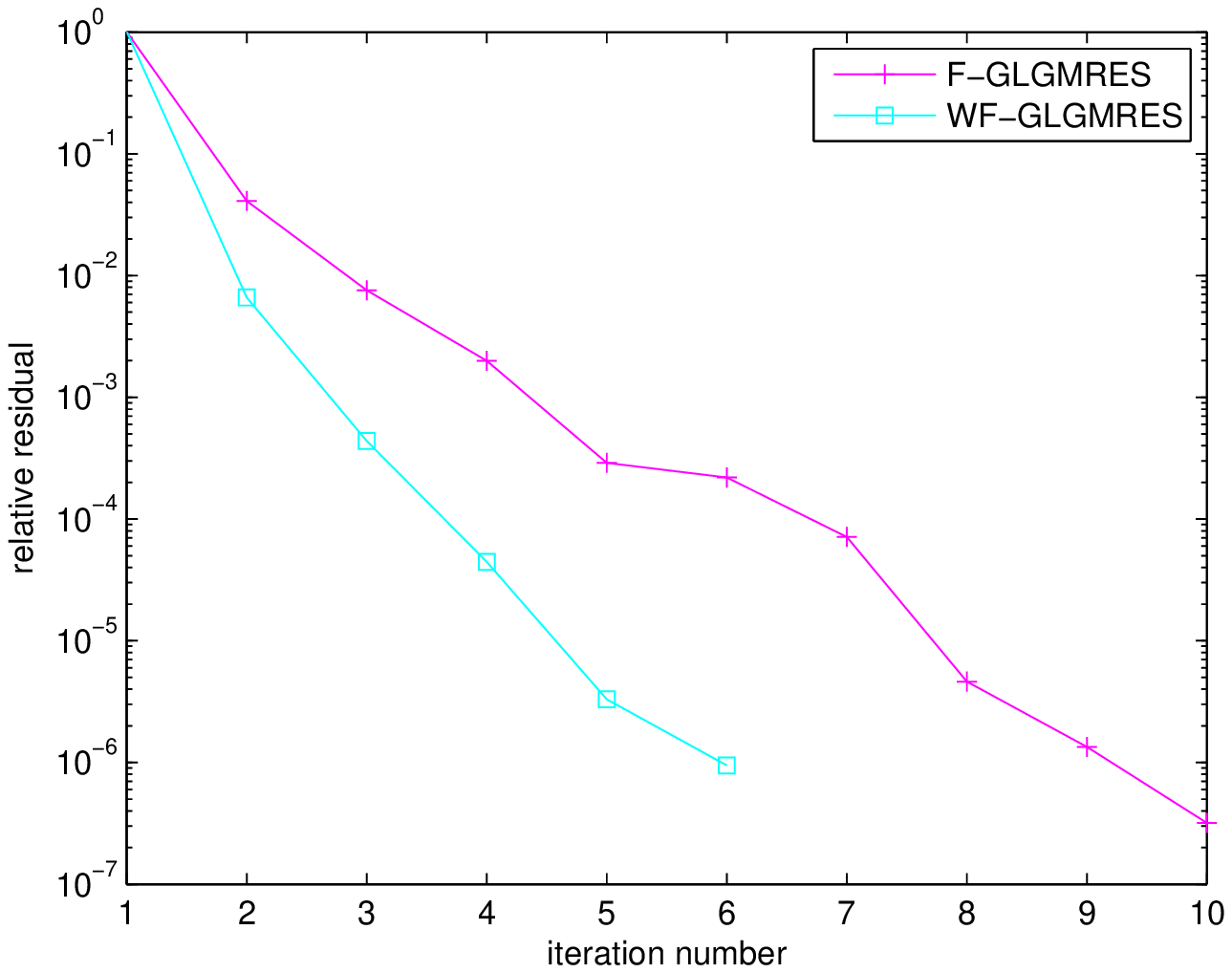}
  \quad
  \includegraphics[width=5cm]{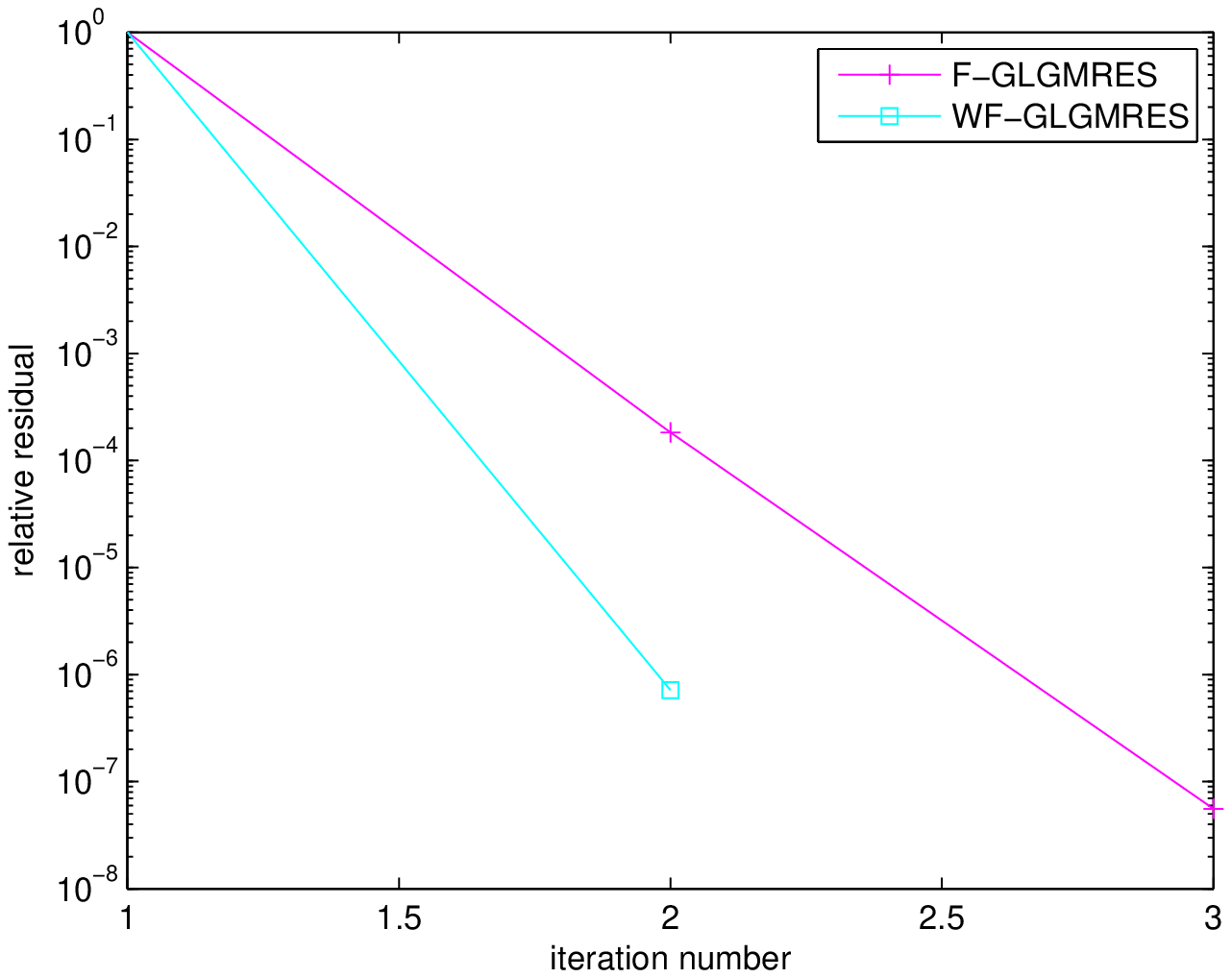}
 \caption{Example 4: Convergence curves of the flexible global GMRES and the weighted flexible global GMRES algorithms. Left: $A=sherman4$, Right: $A=add32$.  }\label{figure6}
\end{figure}
\begin{figure}
\centering

\includegraphics[width=5cm]{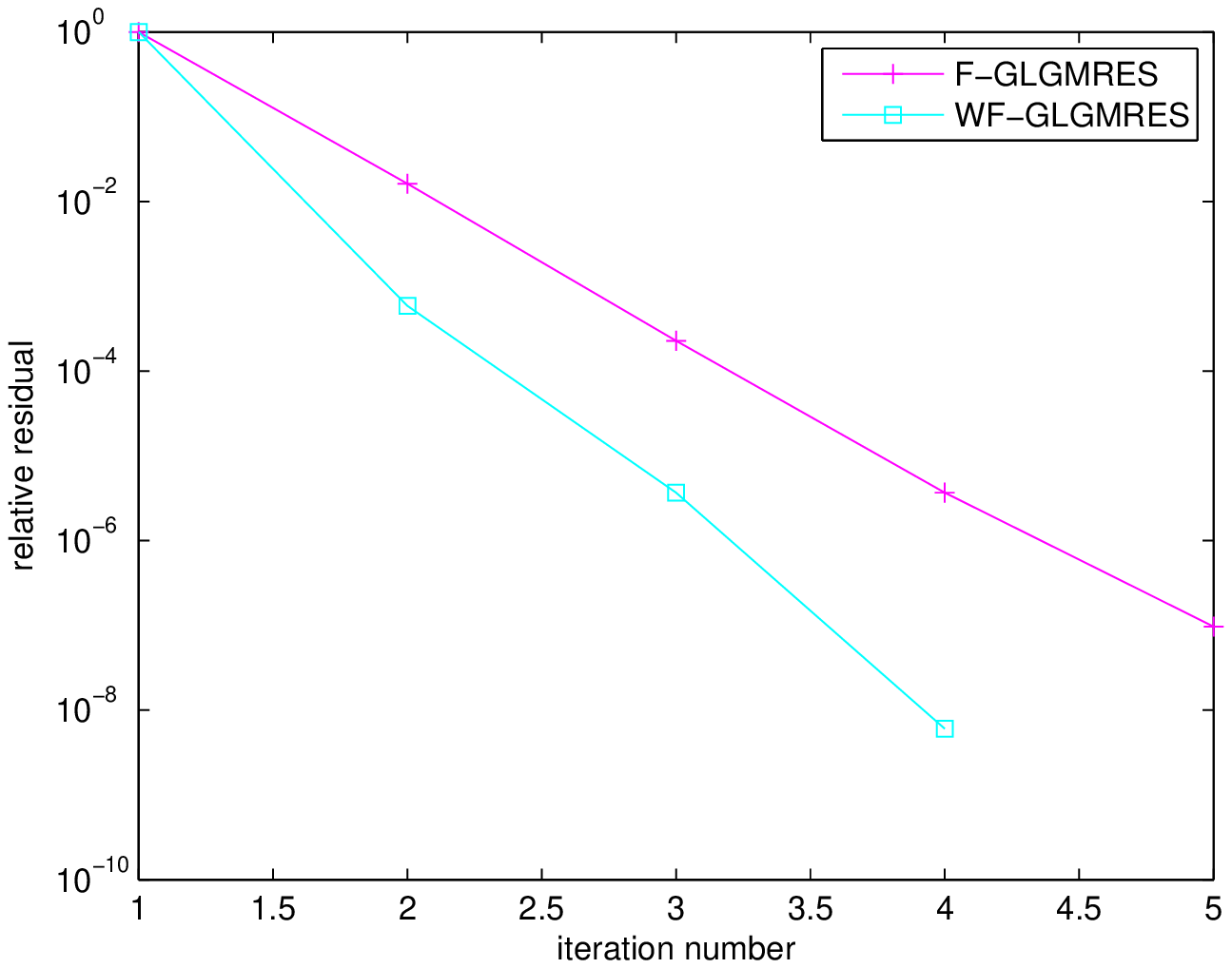}
  \quad
\includegraphics[width=5cm]{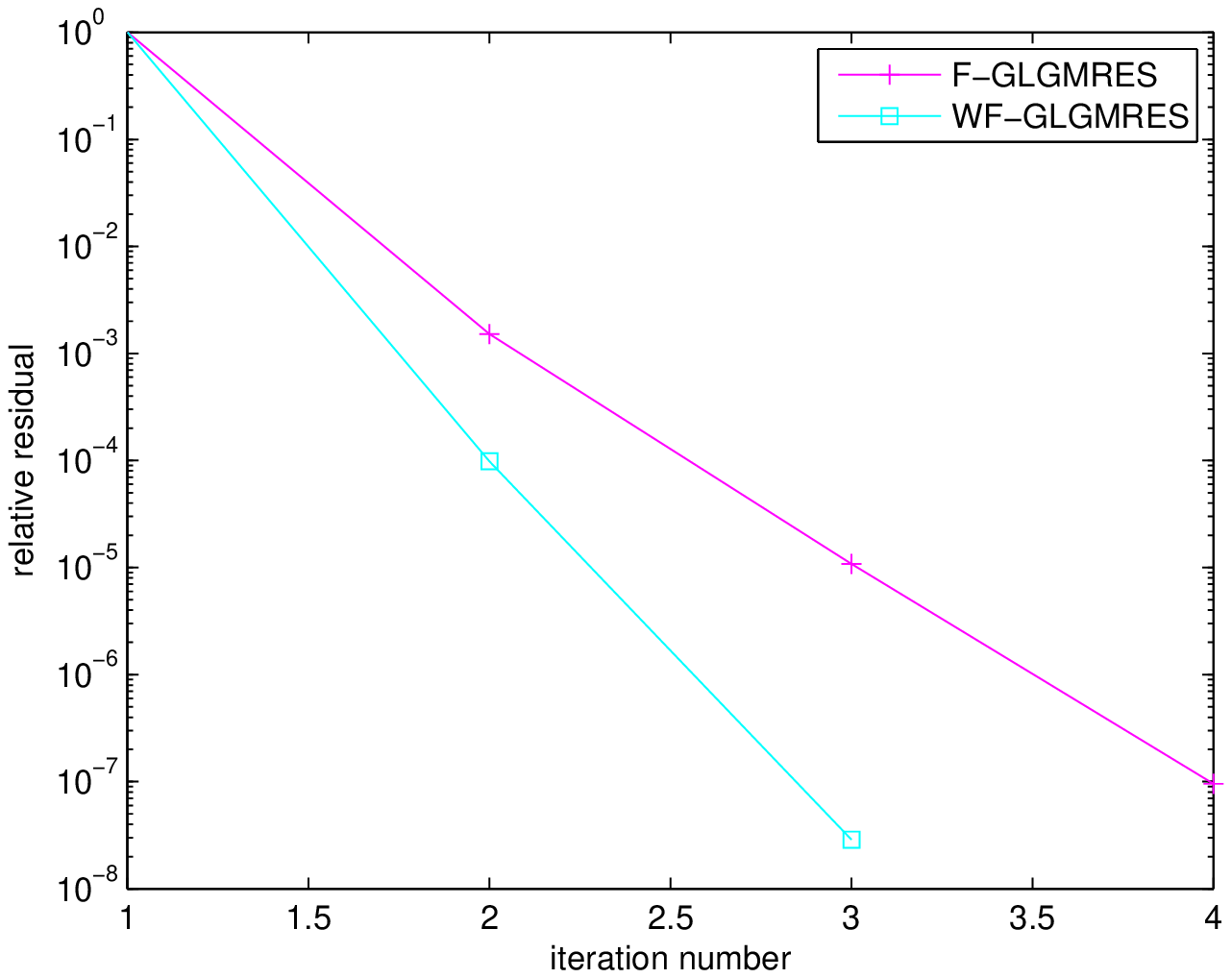}

 \caption{Example 4: Convergence curves of the flexible global GMRES and the weighted flexible global GMRES algorithms. Left: $A=saylr4$, Right: $A=pde2961$.  }\label{figure7}
\end{figure}

Next, we compare the weighted flexible global GMRES with deflation with flexible global GMRES with deflation \cite{Giraud} for solving the five problems, where
$m=20,k=15$ is used. Table \ref{tab6} lists the numerical results. Figures \ref{figure8}--\ref{figure9} plot the convergence curves of the two algorithms during iterations.
Again, it is obvious to see that the weighted algorithm is better than the standard one in terms of iteration numbers and CPU time.
Compared with the numerical results given in Table \ref{tab4}, we find that the flexible and deflated algorithms often need fewer iterations than the deflated versions, however, the CPU time of the former can be much more than the latter. As we have pointed out before, this is due to the fact that the inner iterations bring us a large amount of computational overhead. How to reduce the high cost from inner iterations is beyond the scope of this paper, but deserves further investigation.

Moreover, the two flexible and deflated algorithm still do not work for Sherman2, just like the bare flexible algorithms. One reason is that only five steps of full GLGMRES for solving the linear systems in the inner iterations is not enough for this problem.
Thus, we suggest to use deflated global GMRES when $s$, the number of columns of $C$, is large, say, more than one hundred. On the other hand, when $s$ is of medium-sized, we recommend to use the flexible and deflated global GMRES algorithm.

\begin{table}
\centering
\caption{Example 4: A comparison of flexible W-GLGMRES-D and flexible GLGMRES-D.}\label{tab6}
\begin{tabular}{c|c|c c c}
& & \multicolumn{3}{|c}{ \scriptsize{$m=20,k=15$}}  \\
\cline{3-5}
\scriptsize{Problem} &\scriptsize{Algorithm}&\scriptsize{iter} &\scriptsize{res.norm} & \scriptsize{CPU}\\
\hline
\scriptsize{$A=sherman4$} & \scriptsize{F-GLGMRES-D}&\scriptsize{7}&\scriptsize{8.39783e-08}& \scriptsize{513.5553}\\
\scriptsize{$B=fdm(\cos(xy),e^{y^2x},100)$} & \scriptsize{}&\scriptsize{}&\scriptsize{}& \scriptsize{} \\
\scriptsize{$n=1104,s=400$} &  \scriptsize{WF-GLGMRES-D}& \scriptsize{5}&\scriptsize{8.0266e-07}& \scriptsize{ 370.8139}\\
\hline
\scriptsize{$A=add32$} & \scriptsize{F-GLGMRES-D}&\scriptsize{4}&\scriptsize{1.0656e-08}&\scriptsize{1.2597e+03}\\
\scriptsize{$B=fdm(\sin(xy),e^{xy},10)$} & \scriptsize{}&\scriptsize{}&\scriptsize{}& \scriptsize{} \\
\scriptsize{$n=4960,s=400$} & \scriptsize{WF-GLGMRES-D}&\scriptsize{3}& \scriptsize{8.7016e-12} & \scriptsize{928.9309} \\
\hline
\scriptsize{$A=saylr4$} & \scriptsize{F-GLGMRES-D}&\scriptsize{5}&\scriptsize{5.9416e-08}&\scriptsize{1.0732e+03}\\
\scriptsize{$B=fdm(\sin(xy),e^{xy},10)$} & \scriptsize{}&\scriptsize{}&\scriptsize{}& \scriptsize{} \\
\scriptsize{$n=3564,s=400$} & \scriptsize{WF-GLGMRES-D}&\scriptsize{3} &\scriptsize{2.1611e-07} &\scriptsize{640.1040}\\
\hline
\scriptsize{$A=sherman2$} & \scriptsize{F-GLGMRES-D}&\scriptsize{$\uparrow 2500$}& \scriptsize{$\dagger$} & \scriptsize{-}\\
\scriptsize{$B=fdm(\sin(xy),e^{xy},10)$} & \scriptsize{}&\scriptsize{}&\scriptsize{}& \scriptsize{} \\
\scriptsize{$n=1080,s=400$} & \scriptsize{WF-GLGMRES-D}&\scriptsize{$\uparrow 2500$}& \scriptsize{$\dagger$} & \scriptsize{-}\\
\hline
\scriptsize{$A=pde2961$} & \scriptsize{F-GLGMRES-D}&\scriptsize{5}& \scriptsize{2.5480e-08} & \scriptsize{887.0873}\\
\scriptsize{$B=fdm(\sin(xy),e^{xy},10)$} & \scriptsize{}&\scriptsize{}&\scriptsize{}& \scriptsize{} \\
\scriptsize{$n=2961,s=400$} & \scriptsize{WF-GLGMRES-D}&\scriptsize{3}&\scriptsize{4.3535e-07} & \scriptsize{530.0198}\\
\end{tabular}
\end{table}

\begin{figure}
\centering
 \includegraphics[width=5cm]{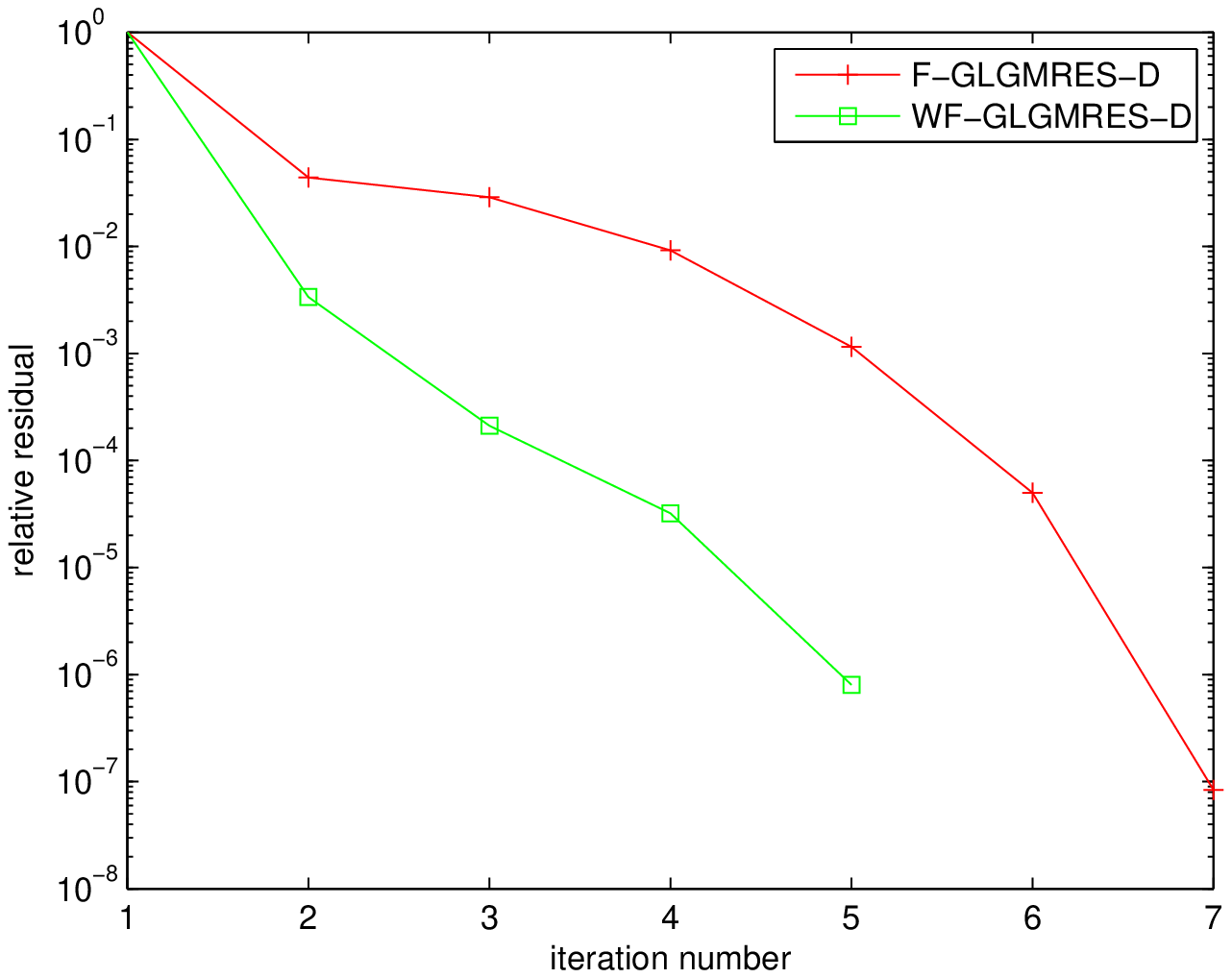}
  \quad
   \includegraphics[width=5cm]{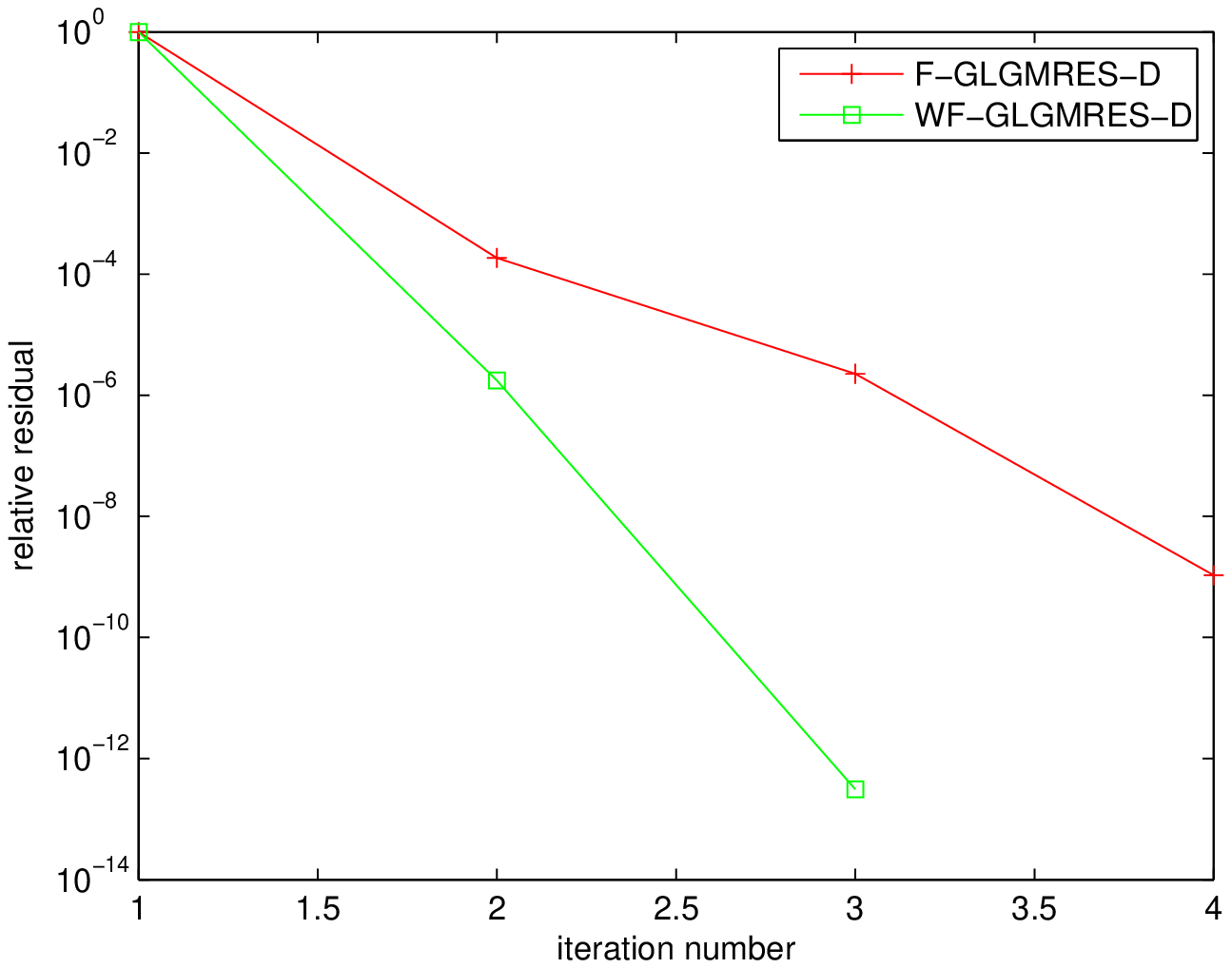}
 \caption{Example 4: Convergence curves of WF-GLGMRES-D and F-GLGMRES-D. Left: $A=sherman4$, Right: $A=add32$.  }\label{figure8}
\end{figure}
\begin{figure}
\centering
 \includegraphics[width=5cm]{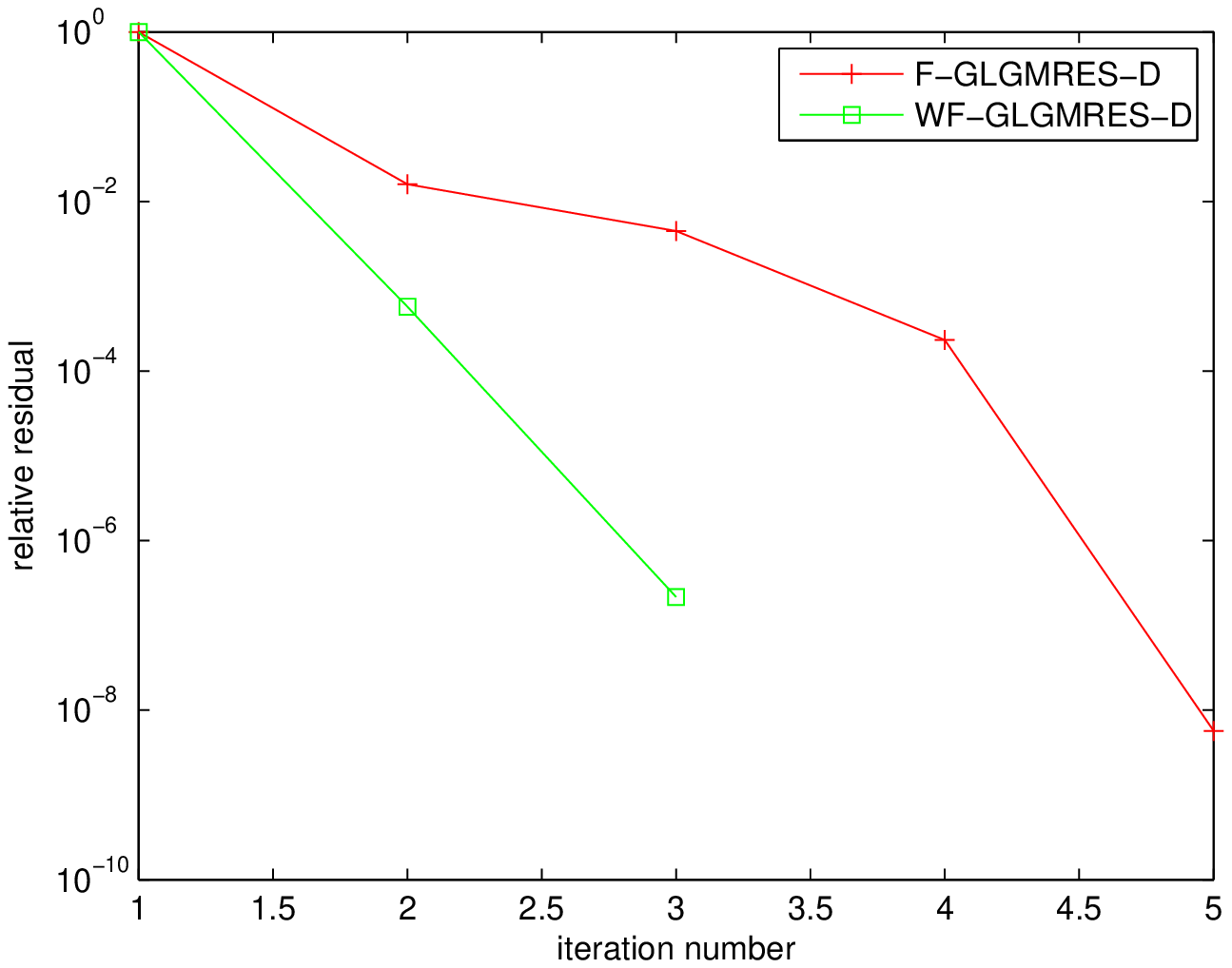}
  \quad
 \includegraphics[width=5cm]{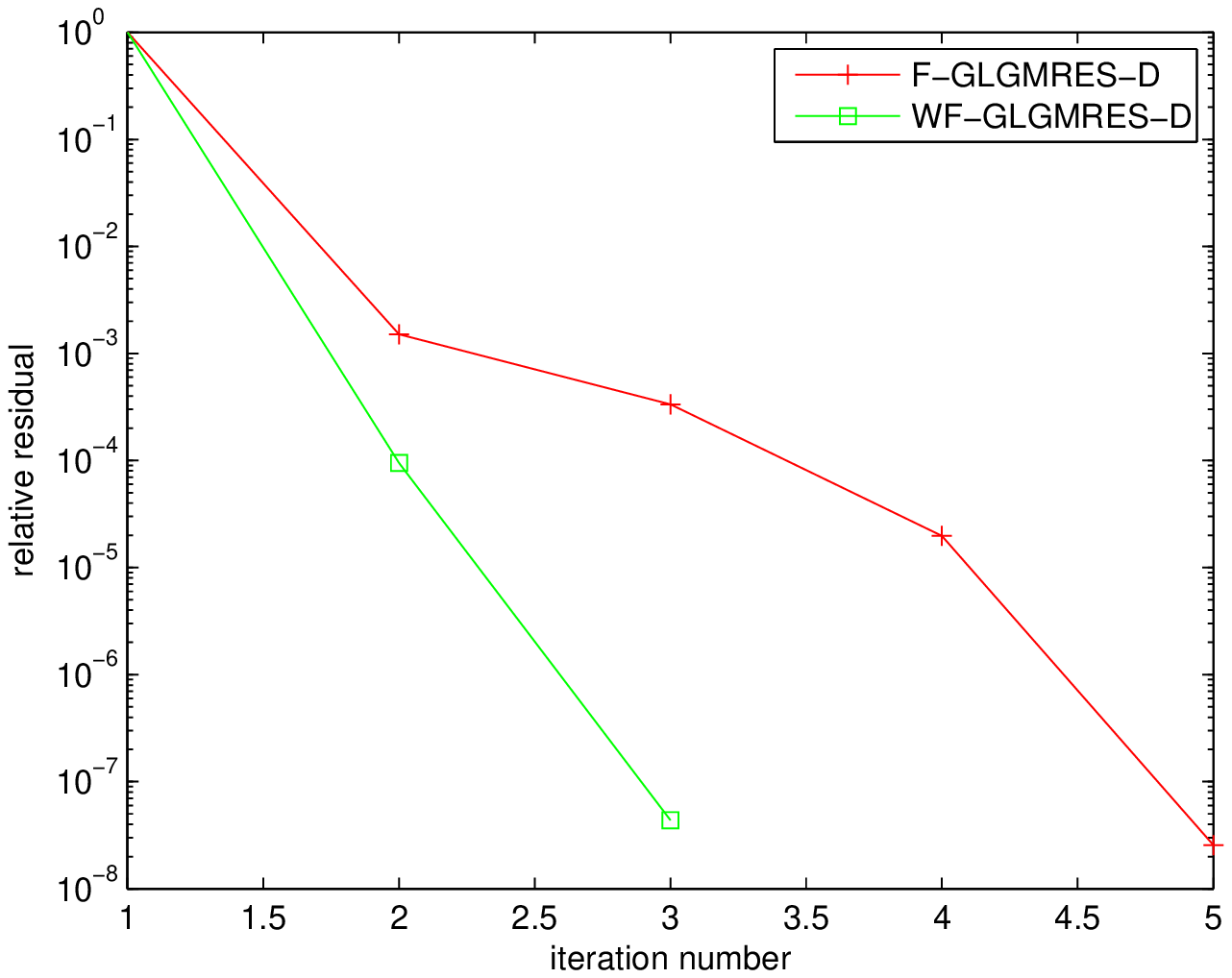}
 \caption{Example 4: Convergence curves of WF-GLGMRES-D and F-GLGMRES-D. Left: $A=saylr4$,  Right: $A=pde2961$.  }\label{figure9}
\end{figure}

\section{Conclusion}
The global GMRES algorithm is popular for large Sylvester matrix equations. The weighting strategy can improve the algorithm by alienating the eigenvalues that obstacle the convergence. However, the optimal choice of the weighting matrix is still an open problem and needs further investigation. Moreover, due to the growth of memory requirements and computational
cost, it is necessary to restart the algorithm efficiently. 

The contribution of this work is three-fold. First, we present three
new schemes based on residual to update the weighting matrix during iterations, and propose a weighted global GMRES algorithm. Second, we apply the deflated restarting strategy to the weighted algorithm, and propose a weighted global GMRES algorithm with deflation for solving large Sylvester matrix equations. Third, we show that in the weighted and deflated global GMRES algorithm, there is no need to change the inner product with respect to diagonal matrix to that with non-diagonal matrix, and our scheme is much cheaper than the one proposed in weighted GMRES-DR algorithm \cite{Embree}.
Further, we consider other acceleration technology such as the flexible preconditioning strategy. For the weighted flexible global GMRES algorithm with deflation, it is interesting to reduce the high cost from inner iterations, and it is definitely a part of our future work.


\begin{thebibliography}{10}

\bibitem{Agoujil} {\sc S. Agoujil, A. H. Bentbib, K. Jabilou and E. M. Sadek},
{\em A minimal residual norm method for large-scale Sylvester matrix equations,
 Electron. Tran. Numer. Anal.}
\textbf{43} (2014),  45--59.


\bibitem{Benner}{\sc P. Benner, R. C. Li, and N. Truhar},
{\em On the ADI method for Sylvester equations,
 J. Comput. Appl. Math.}
\textbf{223} (2009),  1035--1045.

\bibitem{Calvetti}{\sc D. Calvetti},
{\em Application of ADI iterative methods to the restoration of noisy images,
SIAM J. Matrix Anal. Appl.}
\textbf{17} (1996),  165--186.

\bibitem{Datta} {\sc B. N. Datta},
{\em Numerical Methods for Linear Control Systems Design and Analysis, Elsevier Press},
 2003.

\bibitem{Dattakn} {\sc B. N. Datta, K. Datta},
{\em Theoretical and Computational Aspects of Some Linear Algebra Problems in Control Theory, in: C.I. Byrnes, A. Lindquist Eds., Computational and Combinatorial Methods in Systems Theory, Elsevier, Amsterdam},
\textbf{177} (1986),  201--212.

\bibitem{Dehghan} {\sc M. Dehghan, M. Hajarian},
{\em Two algorithms for finding the Hermitian reflexive and skew-Hermitian
solutions of Sylvester matrix equations,
Appl. Math. Lett.}
\textbf{24} (2011),  444--449.

\bibitem{Duan} {\sc C. Duan, Z. Jia},
{\em A global harmonic Arnoldi method for large non-Hermitian eigenproblems with an application to multiple eigenvalue problems, J. Comput. Appl. Math.},
\textbf{234} (2010), 845--860.

\bibitem{Embree}{\sc M. Embree, R. B. Morgan and H. V. Nguyen},
{\em Weighted inner products for GMRES and GMRES-DR,}
 (2017), arXiv:1607.00255v2.

\bibitem{Essai}
{\sc A. Essai},
{\em Weighted FOM and GMRES for solving nonsymmetric linear systems,
 Numer. Alg.}
\textbf{18} (1998),  277--292.

\bibitem{Giraud} {\sc L. Giraud, S. Gratton, X. Pinel and X. Vasseur},
{\em Flexible GMRES with deflated restarting, SIAM J. Sci. Comput.}
\textbf{32} (2010),  1858--1878.

\bibitem{Guennouni} {\sc A. El Guennouni, K. Jbilou, and A. J. Riquet},
{\em Block Krylov subspace methods for solving large Sylvester equations, Numer. Alg.}
\textbf{29} (2002),  75--96.

\bibitem{Guttel}{\sc S. Guttel and J. Pestana},
{\em Some observations on weighted GMRES, Numer. Alg.}
\textbf{67} (2014),  733--752.

\bibitem{Heyouni} {\sc M. Heyouni},
{\em Extended Arnoldi methods for large low-rank Sylvester matrix equations,
 Appl. Numer. Math.}
\textbf{60} (2010),  1171--1182.

\bibitem{Heyouniessai} {\sc M. Heyouni and A. Essai},
{\em Matrix Krylov subspace methods for linear systems with multiple right-hand sides,
Numer. Alg.}
\textbf{40} (2005),  137--156.

\bibitem{Jaimoukha} {\sc I. M. Jaimoukha and E. M. Kasenally},
{\em Krylov subspace methods for solving large Lyapunov equations, SIAM J. Numer. Anal.}
\textbf{31} (1994),  227--251.

\bibitem{Jbilou} {\sc K. Jbilou, A. Messaoudi, H. Sadok},
{\em Global FOM and GMRES algorithms for matrix equations, Appl. Numer. Math.}
\textbf{31} (1999), 49--63.

\bibitem{Riquet} {\sc K. Jbilou, and A. J. Riquet},
{\em Projection methods for large Lyapunov matrix equations, Linear Algebra Appl.}
\textbf{415} (2006),  344--358.

\bibitem{Jiang} {\sc W. Jiang, and G. Wu} {\em A thick-restarted block Arnoldi algorithm with modified Ritz vectors for large eigenproblems,
Comput. Math. Appl.} \textbf{60} (2010), 873--889.

\bibitem{Khorsand}{\sc M. Khorsand Zak and F. Toutounian},
{\em Nested splitting CG-like iterative method for solving the continuous Sylvester equation and preconditioning,}
Adv. Comput. Math.
\textbf{40} (2013),  865--880.

\bibitem{Lin} {\sc Y. Lin},
{\em Minimal residual methods augmented with eigenvectors for solving Sylvester equations and generalized Sylvester equations, Appl. Math. Comput.}
\textbf{181} (2006) 487--499.

\bibitem{Simoncinilin} {\sc Y. Lin, and V. Simoncini},
{\em Minimal residual methods for large scale Lyapunov equations,
 Appl. Numer. Math.}
\textbf{72} (2013) 52--71.

\bibitem{Matrix Market website}
{\sc Matrix Market},
{\em http:// math.nist.gov/ matrixMarket/.} Accessed 2016.

\bibitem{Morganr} {\sc R. Morgan},
{\em GMRES with deflated restarting, SIAM J. Sci. Comput.}
\textbf{24} (1) (2002) 20--37.

\bibitem{Morgan} {\sc R. Morgan}, {\em Restarted block GMRES with deflation of eigenvalues, Appl. Numer. Math.}, \textbf{54} (2005), 222--236.

\bibitem{Mor}
{\sc R. Morgan, and M. Zeng}, {\em A harmonic restarted Arnoldi algorithm for calculating eigenvalues and determining multiplicity, Linear Algebra Appl.} {\textbf 415} (2006)
96--113.

\bibitem{Panjeh} {\sc M. Mohseni Mohgadam and F. Panjeh Ali Beik},
{\em A new weighted global full orthogonalization method for solving nonsymmetric linear systems with multiple right-hand sides, Int. Electron. J. Pure Appl. Math.}
\textbf{2} (2010) 47--67.

\bibitem{Panjehali}{\sc F. Panjeh Ali Beik, and M. Mohseni Mohgadam},
{\em Global generalized minimum residual method for solving Sylvester equation, Aust. J. Basic Appl. Sci.}
\textbf{5} (2011) 1128--1134.

\bibitem{Panjehbeik}{\sc F. Panjeh Ali Beik, and D. Khojasteh Salkuyeh},
{\em Weighted versions of Gl-FOM and Gl-GMRES for solving general coupled linear matrix equations,
 Comput. Math. \& Math. Phy.}
\textbf{55} (2015) 1606--1618.


\bibitem{Penzel} {\sc T. Penzel},
{\em LYAPACK: A MATLAB toolbox for large Lyapunov and Riccati equations, model reduction problems, and linear-quadratic optimal control problems}, software available at
 https://www.tu-chemnitz.de/sfb393/lyapack/.

\bibitem{Saad}{\sc Y. Saad},
{\em A flexible inner-outer preconditioned GMRES Algorithm, SIAM J. Sci. Comput.},
\textbf {14} (1993) 461--469.

\bibitem{Saberinajafi} {\sc H. Saberi Najafi, H. Zareamoghaddam},
{\em A new computational GMRES method, Appl. Math. Comput.}
199 (2008) 527--534.

\bibitem{Simoncini} {\sc V. Simoncini},
{\em A new iterative method for solving large-scale Lyapunov matrix equations, SIAM J. Sci. Comput.}
\textbf{29} (2007) 1268--1288.

\bibitem{Simonciniv} {\sc V. Simoncini},
{\em Computational methods for linear matrix equations, SIAM Rev.}
\textbf{58} (2016) 377--441.

\bibitem{Wu} {\sc G. Wu, Y. Wei},
{\em A Power Arnoldi algorithm for computing PageRank, Numer. Linear Algebra Appl.}
\textbf{14} (2007) 521--546.

\bibitem{KWu} {\sc K. Wu, and H. Simon}, {\em Thick-restart Lanczos method for sysmmetric eigenvalue problems, SIAM J. Matrix Anal. Appl.} \textbf{22} (2000) 602--616.

\bibitem{Zhong} {\sc H. X. Zhong, G. Wu},
{\em Thick restarting the weighted harmonic Arnoldi algorithm for
large interior eigenproblems, Int. J. Comput. Math.}
\textbf {88} (2011) 994--1012.

\end{thebibliography}
\end{document}